\documentclass{article}%
\usepackage{amsmath}
\usepackage{amsfonts}
\usepackage{amssymb}
\usepackage{lscape}

\newtheorem{theorem}{Theorem}
\newtheorem{corollary}[theorem]{Corollary}
\newtheorem{proposition}[theorem]{Proposition}
\newtheorem{lemma}[theorem]{Lemma}
\newenvironment{proof}[1][Proof]{\noindent\textbf{#1.} }{\ \rule{0.5em}{0.5em}}
\newenvironment{remark}{\noindent\textbf{Remark.}\ }{}
\numberwithin{equation}{section}
\numberwithin{theorem}{section}
\newenvironment{t_enumerate}{
\begin{enumerate}
\setlength{\itemsep}{1pt}
\setlength{\parskip}{0pt}
\setlength{\parsep}{0pt}}{\end{enumerate}
}

\newenvironment{remarks}{
\noindent\textbf{Remarks.}
\begin{enumerate}
\setlength{\itemsep}{1pt}
\setlength{\parskip}{0pt}
\setlength{\parsep}{0pt}}{\end{enumerate}
}

\begin{document}

\date{}
\title{Higher Asymptotics of Laplace's Approximation}
\author{William D. Kirwin\thanks{Center for Mathematical Analysis, Geometry and Dynamical Systems, Instituto Sup\'erior T\'ecnico, Av. Rovisco Pais, 1049-001, Lisbon, \textsc{Portugal}; Email: will.kirwin@gmail.com}}
\maketitle

\begin{abstract}
We present expressions for the coefficients which arise in asymptotic expansions of multiple integrals of Laplace type (the first term of which is known as Laplace's approximation) in terms of asymptotic series of the functions in the integrand. Our most general result assumes no smoothness of the functions of the integrand, but the expressions we obtain contain integrals which may be difficult to evaluate in practice. We then make additional assumptions which are sufficient to simplify these integrals, in some cases obtaining explicit formulae for the coefficients in the asymptotic expansions.
\end{abstract}


%%%%%%%%%%%%%%%%%%%%%%%%%%%%%%%%%%%%%%%%%%%%%%%%%%%%%%%%%%%%%%%%%%
\section{Introduction.}
%%%%%%%%%%%%%%%%%%%%%%%%%%%%%%%%%%%%%%%%%%%%%%%%%%%%%%%%%%%%%%%%%%

Consider the integral
\[
\int_{a}^{b}e^{-kf(x)}g(x)\,dx
\]
where $f$ and $g$ are sufficiently smooth functions and $f$ attains a nondegenerate, unique minimum in $[a,b]$ at a point $x_{0}\in(a,b)$. In \cite{Laplace-Prob}\footnote{The approximation (\ref{eqn:Laplace1}) appears in \cite{Laplace-Prob}, Part 2, Chapter 1, Section 27.}, Laplace observed that as $k$ increases, the integral localizes near the minimum of $f$, and in particular showed that as $k\rightarrow\infty$,
\begin{equation}
\int_{a}^{b}e^{-kf}g\,dx\sim\sqrt{\frac{2\pi}{k}}\frac{g(x_{0})}{\sqrt{f^{\prime\prime}(x_{0})}}. \label{eqn:Laplace1}
\end{equation}

In this article we will be concerned with the asymptotics of integrals of the form
\[
\int_{R}e^{-kf}g~d^{d}\mathbf{x},
\]
known as Laplace-type integrals, where $k$ is a real parameter which goes to infinity, $f$ and $g$ are real-valued functions defined on some region $R\subset\mathbb{R}^{d}$, and we assume that $f$ attains a unique minimum in the interior of $R$. For convenience, say this minimum value of $f$ is $0$ at $\mathbf{0}\in R$.

Laplace's idea remains valid in $d$ dimensions; namely, for $k$ sufficiently large, the integral becomes localized near the minimum of $f$ and one may compute what is now known as Laplace's approximation:
\begin{equation}
\int_{R}e^{-kf}g~d^{d}\mathbf{x\sim}\left(  \frac{2\pi}{k}\right)^{d/2}\frac{g(\mathbf{0})}{\sqrt{\det Hf(\mathbf{0})}}, \label{eqn:LaplacesApprox}
\end{equation}
where $Hf(\mathbf{0})$ denotes the Hessian of $f$ evaluated at $\mathbf{0}$. To prove this requires two main ideas: first, one shows that the integral localizes to a neighborhood of $\mathbf{0}$. Second, one applies the Morse Lemma to transform $f$, and hence the integral, to a standard form from which one can deduce the desired approximation. We refer the interested reader to \cite{Bleistein-Handelsman}, \cite{deBruijn}, and \cite{Wong} for some modern treatments of Laplace's approximation and related ideas.

When the function $f$ attains a unique minimum somewhere on the boundary, the integral localizes near this boundary point as $k\rightarrow\infty$, but the analysis and resulting asymptotics are slightly different. We do not consider this case in this article.

There are several important and useful generalizations of Laplace's approximation. When $f$ is purely imaginary, because of cancelations arising from oscillations, the integral localizes near critical points of $f$ (not just minima), and the resulting Laplace-type approximations are called stationary phase approximations. If $f$ is defined on a region in $\mathbb{C}$ and is complex valued, the analysis takes on quite a different character, and Laplace-type approximations are known as the method of steepest descent.

All of these Laplace-type approximations are very useful in applications. For example, Laplace and stationary phase approximations arise naturally in statistical mechanics and quantum field theory because the partition functions tend to be of Laplace type. In the quantum case, the natural parameter is $i/\hslash$ and one is interested in the semiclassical limit $\hslash\rightarrow0$, whereas in statistical mechanics (for example in the microcanonical ensemble) one is typically interested in the joint limit as the inverse temperature (which plays the role of $k$) and the dimension $d$ both tend to infinity.

Laplace-type approximations also play a role in pure mathematics; equivariant localization is an application of stationary phase to certain integrals of equivariant cohomology classes where the resulting asymptotic expansion is, somewhat magically, exact. In geometric quantization, Laplace's approximation has been used by the author and Brian C. Hall \cite{Hall-K} to analyze the non-unitarity of ``quantization commutes with reduction.''\footnote{``Quantization commutes with reduction'' is a result of Guillemin and Sternberg \cite{Guillemin-Sternberg} which says that, at the level of vector spaces, the geometric quantization of the symplectic reduction of a K\"{a}hler $G$-manifold is isomorphic to the subspace of $G$-invariant vectors of the quantization of $M$ itself. These spaces, though isomorphic, are not naturally unitary.} This list is by no means exhaustive or unbiased, and we urge the interested reader to consult the above-mentioned references for further details and applications.

\bigskip
In this article, we will stay in the realm of real-valued functions and study a different extension of Laplace's approximation. In modern terms, Laplace's approximation gives the first term in an asymptotic expansion
\[
\int_{R}e^{-kf}g~d^{d}\mathbf{x\sim}\sum_{j=0}^{N}k^{-(j+d)/2}\zeta_{j}.
\]
This means, by definition, that\footnote{A function $h$ is little-$o$ of $k^{p}$, written $h(k)=o(k^{p}),~k\rightarrow\infty,$ if $\lim_{k\rightarrow\infty}k^{-p}h(k)=0.$ A related notion which we will also use is big-$O$; $h(k)=O(k^{p}),~k\rightarrow\infty$ if there exists a constant $C>0$ such that for $k$ sufficiently large, $\left\vert h(k)\right\vert \,<Ck^{p}.$ It follows that if $h(k)=o(k^{p}),~k\rightarrow\infty$, then $h(k)=O(k^{p})$. Conversely, if $h(k)=O(k^{p+\varepsilon})$ for any $\varepsilon>0$, then $h(k)=o(k^{p})$.}
\[
\int_{R}e^{-kf}g\,d^{d}\mathbf{x}=\sum_{j=0}^{N}k^{-(j+d)/2}\zeta_{j}+o(k^{-(N+d)/2}).
\]

The main result of this paper is Theorem \ref{thm:main}, which gives expressions for $\zeta_{j}$ under quite general hypotheses; in particular, we make no smoothness assumptions on $f$ or $g$, assuming only that they admit asymptotic expansions as $\left\vert\mathbf{x}\right\vert \rightarrow0$. (The order of the expansions of $f$ and $g$ determines the order $N$ of the asymptotic expansion of the integral.) The function $g$ is even permitted a mild singularity at $\mathbf{0}$. The proof is essentially a modification of an existence proof of Fulks and Sather \cite{Fulks-Sather} using Theorem \ref{thm:Frame}, which is a variant of a result of Frame \cite{Frame} regarding inversions of series.

The expressions for the coefficient which appear in Theorem \ref{thm:main}, though, involve certain integrals which may in general be difficult to evaluate. With more restrictive assumptions, we can apply Theorem \ref{thm:main} to go much further, obtaining Corollaries \ref{cor:f0const} and \ref{cor:taylor} and Theorem \ref{thm:best}). These more restrictive assumptions are still sufficiently general to include most common applications (though certainly not all). In particular, we make two (partially) independent simplifications.

The first simplification occurs if we assume that $f/\left\vert \mathbf{x}\right\vert ^{\nu}$ is continuous at $\mathbf{x}=\mathbf{0}$, where $\nu>0$ is the order of the zero of $f$ at $\mathbf{0}$ (\textit{c.f.} Corollary \ref{cor:f0const}). In this case, the integrals appearing in Theorem \ref{thm:main} simplify greatly, and given the appropriate data about $f$ and $g$, one expects that the resulting integrals \emph{can} be easily evaluated.

The second simplification occurs if we assume instead that $f$ has a nondegenerate minimum at $\mathbf{0}$. In this case, the Morse Lemma tells us that in a neighborhood of $\mathbf{0}$, the function $f$ can be put into ``standard'' quadratic form. This extra structure allows us to make contact with the usual expressions of Laplace's approximation. Moreover, one can insure (by an appropriate linear change of coordinates) that $f/\left\vert \mathbf{x}\right\vert^{2}$ is continuous at $\mathbf{0}$, and thus our previous simplification applies. The result is the following theorem (we emphasize that the main result of this paper is Theorem \ref{thm:main}; we quote the following theorem because it involves only Taylor coefficients and combinatorial quantities, and so is in some sense our most ``explicit'' result, although it is simply a corollary of Theorem \ref{thm:main}).

\begin{theorem}
\label{thm:best}Let $R\subset\mathbb{R}^{d}$ be a measurable set which contains $\mathbf{0}$ as an interior point and suppose $f\in C^{N+2}(R)$ and $g\in C^{N}(R)$. Suppose moreover that $f$ has a unique, nondegenerate minimum value of $0$ at $\mathbf{0}$. Assume that for some $k_{0}>0$, the integral $\int_{R}e^{-k_{0}f}g\,d^{d}\mathbf{x}$ converges. Then there exists a linear transformation $P:\mathbb{R}^{d}\rightarrow\mathbb{R}^{d}$ and an asymptotic expansion
\begin{equation}
\int_{R}e^{-kf}g\,d^{d}\mathbf{x}=k^{-d/2}\sum_{j=0}^{\lfloor N/2\rfloor} \zeta_{2j}k^{-j}+O(k^{-(N+d)/2+1}),~k\rightarrow\infty
\label{eqn:asymptotic_expansion}
\end{equation}
where $\lfloor N/2\rfloor$ denotes the largest integer less than $N/2$, and the coefficients are given by
\begin{align*}
\zeta_{2j} &  =\frac{(2\pi)^{d/2}}{\sqrt{\det H_{x}f(\mathbf{0})}}\sum_{m=0}^{2j}\sum_{r=1}^{m}\frac{(-1)^{r}}{r!}\\
&  \qquad\times\sum_{\left\vert \beta\right\vert =2j-m}\sum_{\left\{\substack{n_{1}+n_{2}+\cdots+n_{r}=m\\n_{i}\geq1,~i=1,\dots,r}\right\}}
\sum_{\left\vert \alpha_{1}\right\vert =n_{1}+2}\cdots\sum_{\left\vert\alpha_{r}\right\vert =n_{r}+2}\operatorname*{even}(\beta+\alpha_{1}+\cdots+\alpha_{r})\\
&  \qquad\qquad\qquad\qquad\qquad\qquad\qquad\qquad\qquad\qquad\times \frac{(\beta+\alpha_{1}+\cdots+\alpha_{r}-\mathbf{1})!!}{\beta!\alpha_{1}!\cdots\alpha_{r}!} D_{y}^{\beta}g(\mathbf{0})D_{y}^{\alpha_{1}}f(\mathbf{0})\cdots D_{y}^{\alpha_{r}}f(\mathbf{0})
\end{align*}
in which $H_{x}f(\mathbf{0})$ is the Hessian of $f$ at $\mathbf{x}=\mathbf{0}$ (with respect to $\mathbf{x}$), $\mathbf{y}=P\mathbf{x}$ are coordinates in which $H_y f(\mathbf{0})=\mathbf{1}_{\mathbb{R}^n}$, and for a multi-index $\alpha=(\alpha^{1},\alpha^{2},\dots,\alpha^{d})$, $D_{y}^{\alpha}g$ denotes the mixed partial derivative of order $\alpha$ of $g$ with respect to $\mathbf{y},$ and
\[
\operatorname*{even}(\alpha):=\begin{cases}
0 & \text{ if, for any }m\text{, }\alpha^{m}\text{ is odd}\\
1 & \text{ otherwise.}
\end{cases}
\]
Empty sums, $0!$ and $(-1)!!$ are all understood to be $1$.
\end{theorem}

Recall that a critical point of $f$ is said to be nondegenerate if the Hessian at that point is invertible. The coordinate change $\mathbf{y}=P\mathbf{x}$ appearing in Theorem \ref{thm:best} can be computed explicitly by diagonalizing the Hessian of $f$. More precisely, since it is symmetric, we can write $Hf(\mathbf{0})=Q^{-1}DQ$ for some orthogonal matrix $Q$ and some diagonal matrix $D$. Since the minimum of $f$ at $\mathbf{0}$ is nondegenerate, the eigenvalues of $Hf(\mathbf{0})$ are all positive. The matrix $P$ of Theorem \ref{thm:best} is then $P=\sqrt{D}Q$.

We give expanded expressions of the first few coefficients in the appendix.

\bigskip

Our results are not the first of their kind. It has been known for some time that with sufficient hypotheses on $f$ and $g$, such an expansion exists, and various methods for computing the coefficients have been given. In \cite{Bleistein-Handelsman}, Bleistein and Handelsman assume that the minimum of $f$ is nondegenerate to obtain a complete asymptotic expansion of the form (\ref{eqn:asymptotic_expansion}). The coefficients, though, are computed in terms of a Jacobian of a coordinate transformation (which essentially arises from the Morse Lemma), that in general cannot be computed explicitly. Nevertheless, since the result is eventually evaluated at $x=0$, it is possible to proceed term by term, obtaining explicit formulas in terms of the derivatives of $f$ and $g$.

In a similar direction, in \cite{Skinner}, assuming $f(x)$ has a nondegenerate minimum, Skinner makes a coordinate change (which diagonalizes the Hessian of $f(x)$ at the minimum) and computes $\zeta_2$. As he points out, his methods could likely be used to directly compute the higher coefficients in this case. The end result would be essentially Theorem \ref{thm:best} quoted above.

In dimension one and assuming $f=\sum_{j=2}^{\infty}a_{j}x^{j},~a_{2}\neq0$ and $g=\sum_{j=0}^{\infty}b_{j}x^{j}$, de Bruijn \cite[(4.4.9)]{deBruijn} gives a complete asymptotic expansion of $\int_{-\infty}^{\infty}e^{-kf}g\,dx$ in terms of the coefficients of the series expansion of the product $g\exp\{-k\sum_{j=3}^{\infty}a_{j}x^{j}\}$. This series expansion of the product can in principle be computed in terms of the $a_{j}$ and $b_{j}$ using Fa\`{a} di Bruno's formula, though de Bruijn does not do this.

In the $1$-dimensional case, Wojdylo gives a closed form expression for $\zeta_{j}$ for all $j$ \cite{WojdyloSIAM, WojdyloJCAM}. His expressions should correspond to de Bruijn's expansions in terms of Taylor series with the above  mentioned simplifications using Fa\`{a} di Bruno's formula. Our results are in fact a generalization of those of Wojdylo in the precise sense that when $d=1$, our formulas reduce to his (see Section \ref{sec:1D}).

More recently, Denef and Sargos \cite{Denef-Sargos}, Kaminksi and Paris \cite{Kaminski-ParisI, Kaminski-ParisII}, and Liakhovetski and Paris \cite{Liakhovetski-Paris} use the relationship between polynomials and Newton polygons to analyze the asymptotics when the exponent $f$ is a polynomial, possibly with noninteger powers. In a related direction, Dostal and Gaveau \cite{Dostal-Gaveau} also work with the Newton polygon to study stationary phase for a polynomial phase function. Our methods are not generally well adapted to the case of polynomials with noninteger powers since one must essentially make a Taylor expansion of the polynomial, at which point it is difficult to see the geometry of the Newton polygon.

\bigskip
\noindent The rest of the paper is organized as follows. In Section \ref{sec:general} we prove our most general result, Theorem \ref{thm:main}. For the proof, we recall and prove a variant of a result of Frame \cite{Frame} on inversion of series, and summarize in Lemma \ref{lemma:Fulks-Sather} the results of Fulks and Sather \cite{Fulks-Sather} which we need. Section \ref{sec:general} concludes with a few brief remarks regarding some of the combinatorial quantities which arise in Theorem \ref{thm:main}.

In Section \ref{sec:f0consttaylor}, we describe the simplifications that occur in Theorem \ref{thm:main} when we assume that $f/\left\vert \mathbf{x}\right\vert ^{\nu}$ is continuous at $\mathbf{0}$ for some (maximal) $\nu>0$. Then, in addition to the continuity assumption at $\mathbf{0}$, we place additional smoothness hypotheses on $f$ and $g$, and give a corollary to Theorem \ref{thm:main} which expresses the coefficients $\zeta_{j}$ in terms of the Taylor coefficients of $f$ and $g$.

In Section \ref{sec:f0constv=2}, we suppose that the minimum of $f$ is nondegenerate. First, we exploit the nondegeneracy of the minimum to show that by a linear change of coordinates, we can insure that $f/\left\vert\mathbf{y}\right\vert^{2}$ is continuous at $\mathbf{0}$. We then use the results of Section \ref{sec:f0consttaylor} to obtain Theorem \ref{thm:best}. Finally, we relate Theorem \ref{thm:main} to Laplace's approximation (\ref{eqn:LaplacesApprox}).

In Section \ref{sec:1D}, we show that in dimension one our results reduce to those of Wojdylo \cite{WojdyloSIAM, WojdyloJCAM}.

The Appendix contains the first few coefficients of the asymptotic expansion under the most general hypotheses, then again with the assumption that the minimum is nondegenerate.

%%%%%%%%%%%%%%%%%%%%%%%%%%%%%%%%%%%%%%%%%%%%%%%%%%%%%%%%%%%%%%%%%%
\section{The general case.\label{sec:general}}
%%%%%%%%%%%%%%%%%%%%%%%%%%%%%%%%%%%%%%%%%%%%%%%%%%%%%%%%%%%%%%%%%%

In this section, we state and prove our most general result: expressions for the coefficients which appear in the asymptotic expansion of integrals of Laplace-type with no smoothness or nondegeneracy assumptions on the functions appearing in the integrand. As mentioned in the introduction, the price we pay for this generality is that the coefficients are expressed in terms of certain integrals which may be difficult to evaluate in practice. In Section \ref{sec:corollaries} below, we discuss various additional hypotheses which are sufficient to alleviate this problem. Our main result, Theorem \ref{thm:main} below, is essentially a modification of an existence theorem of Fulks and Sather \cite{Fulks-Sather} using a variant, Theorem \ref{thm:Frame}, of a result on inversion of series due to Frame \cite{Frame}.

\bigskip

Let $\mathbf{x}=(x^{1},\dots,x^{d})$ be coordinates on $\mathbb{R}^{d}$. Denote by $S^{d-1}=\{\left\vert \mathbf{x}\right\vert =1\}\subset\mathbb{R}^{d}$ the unit sphere and introduce spherical coordinates $\rho:=\sqrt{(x^{1})^{2}+\cdots+(x^{d})^{2}}$ and $\Omega=\mathbf{x}/\rho\in S^{d-1}.$ Our main result is the following theorem.

\begin{theorem}
\label{thm:main}Suppose $f$ and $g$ are measurable functions on a measurable set $R\in\mathbb{R}^{d}$ which contains $\mathbf{0}$ as an interior point. Suppose further that $f$ attains its unique minimum value of $0$ at $\mathbf{0}\in R$ and is otherwise bounded away from zero, and that there is a positive integer $N$ and
\begin{enumerate}
\item $N+1$ continuous functions $f_{j}(\Omega),~j=0,\dots,N$ with $f_{0}(\Omega)>0$ such that for some real number $\nu>0$
\begin{equation}
f(\rho,\Omega)=\rho^{\nu}\sum_{j=0}^{N}f_{j}(\Omega)\rho^{j}+o(\rho^{N+\nu})\text{ as }\rho\rightarrow0,\text{ and} \label{eqn:f-series}
\end{equation}

\item $N+1$ functions $g_{j}(\Omega),~j=0,\dots,N$ such that for some real number $\lambda>0$
\begin{equation}
g(\rho,\Omega)=\rho^{\lambda-d}\sum_{j=0}^{N}g_{j}(\Omega)\rho^{j}+o(\rho^{N+\lambda-d})\text{ as }\rho\rightarrow0. \label{eqn:g-series}
\end{equation}
Then if there exists $k_{0}>0$ such that
\[
\int_{R}e^{-k_{0}f}g\,d^{d}\mathbf{x}
\]
converges, then there exists an asymptotic expansion
\begin{equation}
\int_{R}e^{-kf}g\,d^{d}\mathbf{x}=\sum_{j=0}^{N}\zeta_{j}k^{-(j+\lambda)/\nu}+o(k^{-(N+\lambda)/\nu}),~k\rightarrow\infty\label{eqn:asympexp}
\end{equation}
where the coefficients are given by
\begin{equation}
\zeta_{j}=\tfrac{1}{\nu}\Gamma\left(  \tfrac{j+\lambda}{\nu}\right)\int_{S^{d-1}}\left[  f_{0}(\Omega)^{-(j+\lambda)/\nu}\sum_{m=0}^{j} g_{j-m}(\Omega)\sum_{r=1}^{m}\binom{-\frac{j+\lambda}{\nu}}{r}\frac{f_{m}^{(r)}(\Omega)}{f_{0}(\Omega)^{r}}\right]  d\Omega, \label{eqn:gencoeffs}
\end{equation}
where
\[
f_{m}^{(r)}(\Omega)=\sum_{\substack{n_{1}+n_{2}+\cdots+n_{r} =m\\n_{i}\geq1,~i=1,\dots,r}}f_{n_{1}}(\Omega)f_{n_{2}}(\Omega)\cdots f_{n_{r}}(\Omega)
\]
is the sum of all ordered products\footnote{For example, $f_{6}^{(3)}=6f_{1}f_{2}f_{3}+3f_{1}^{2}f_{4}+f_{2}^{3}$.} of $r$ elements of $\{f_{1}(\Omega),f_{2}(\Omega),\dots,f_{N}(\Omega)\}$ such that the subscripts add to $m,$ and $\binom{\alpha}{r}:=\alpha(\alpha-1)\cdots(\alpha-r+1)/r!$. Empty sums are understood to be $1$.
\end{enumerate}
\end{theorem}

The first few coefficients of Theorem \ref{thm:main} are listed in the appendix.

\bigskip

\begin{remarks}
\item If we make the slightly stronger hypothesis that $g$ admits an asymptotic expansion of the form
\begin{equation}
\label{eqn:ghyp}
g(\rho,\Omega)=\rho^{\lambda-d}\sum_{j=0}^N g_j(\Omega)\rho^j + O(\rho^{N+\lambda-d+\beta})
\end{equation}
for some $\beta>0$, it is straightforward to make the appropriate modifications to the proof to show that there exists an asymptotic expansion
\begin{equation}
\int_{R}e^{-kf}g\,d^{d}\mathbf{x}=\sum_{j=0}^{N}\zeta_{j}k^{-(j+\lambda)/\nu} +O(k^{-(N+\lambda+\min\{1,\beta\})/\nu}),~k\rightarrow\infty\label{eqn:asympexpB}
\end{equation}

\item If we assume that $f/\rho^{\nu}$ and $g/\rho^{\lambda-d}$ are $C^{N}$ at the origin, then by continuity we have that
\begin{align*}
\left.\partial_{\rho}\left(  \frac{f}{\rho^{\nu}}\right)(\rho,\Omega)\right\vert _{\rho=0}  &  =\lim_{\rho\rightarrow0}\frac{\left(\frac{f}{\rho^{\nu}}\right)  _{0}-\left(  \frac{f}{\rho^{\nu}}\right)_{-\rho\Omega}}{\rho}\\
&  =-\lim_{\rho\rightarrow0}\frac{\left(  \frac{f}{\rho^{\nu}}\right)_{\rho(-\Omega)}-\left(  \frac{f}{\rho^{\nu}}\right)  _{0}}{\rho} =-\left.\partial_{\rho}\left(  \frac{f}{\rho^{\nu}}\right)  (\rho,-\Omega)\right\vert_{\rho=0}
\end{align*}
(similarly for $g$) which implies $f_{j}(-\Omega)=(-1)^{j}f_{j}(\Omega)$ and $g_{j}(-\Omega)=(-1)^{j}g_{j}(\Omega)$. If we let
\[
\eta_{j}(\Omega):=f_{0}(\Omega)^{-(j+\lambda)/\nu}\sum_{m=0}^{j}g_{j-m}(\Omega)\sum_{r=1}^{m} \binom{-\frac{j+\lambda}{\nu}}{r}\frac{f_{m}^{(r)}(\Omega)}{f_{0}(\Omega)^{r}}
\]
denote the integrand in (\ref{eqn:gencoeffs}), then a short computation shows that
\[
\eta_{j}(-\Omega)=(-1)^{j}\eta_{j}(\Omega).
\]
It follows then that $\zeta_{2j+1}=0$ since it is the integral of the antisymmetric function $\eta_{j}$ over the $(d-1)$-sphere.
\end{remarks}

\bigskip

Our proof of Theorem \ref{thm:main} roughly follows that of Fulks and Sather \cite{Fulks-Sather}, although we will use a variant of a formula for the inversion of a series due to Frame \cite{Frame}. There are other, equivalent, formulations of series inversion, see for example Kamber \cite{Kamber}; we choose to use Frame's results because they are particularly simple and well suited to our purposes. Because the statement of Frame's result that we need is slightly different from the original, and Frame's proof actually holds in slightly more generality than it was originally stated, we include the result and the relatively short proof here for completeness.

\begin{theorem}
\label{thm:Frame}\cite[Thm. 1]{Frame} For any real numbers $\nu\neq0$ and $q>0$, let $u=f(\rho)$ and $\rho=f^{-1}(u)$ be inverse functions defined for $\rho$ near $0$ by the convergent power series
\begin{align}
u^{\nu}  &  =\rho^{\nu}\sum_{j=0}^{\infty}a_{j}\rho^{j},\text{ with } a_{0}>0\text{, and}\label{eqn:Frame_series}\\
\rho^{q}  &  =u^{q}\sum_{j=0}^{\infty}b_{j}u^{j}. \label{eqn:Frame_inv_series}
\end{align}
Then the coefficients $b_{j}$ in the inverted power series (\ref{eqn:Frame_inv_series}) are given explicitly in terms of the coefficients $a_{j}$ of (\ref{eqn:Frame_series}) by the inversion formula
\begin{equation}
b_{j}=\frac{q}{j+q}\,a_{0}^{-(j+q)/\nu}\sum_{r=1}^{j}\binom{-\frac{j+q}{\nu}}{r}a_{0}^{-r}a_{j}^{(r)}\text{ for }j>0, \label{eqn:inv_coeffs}
\end{equation}
where empty sums are understood to be $1$, $\binom{\alpha}{r}:=\alpha(\alpha-1)\cdots(\alpha-r+1)/r!,$ and
\[
a_{j}^{(r)}=\sum_{\substack{n_{1}+n_{2}+\cdots+n_{r}=j\\n_{i}\geq1,~i=1,\dots,r}}a_{n_{1}}a_{n_{2}}\cdots a_{n_{r}}
\]
is the sum of all ordered products of $r$ terms of the set $\{a_{1},a_{2},a_{3},\dots\}$ in which the sum of subscripts is $k$.
\end{theorem}

\begin{proof}
First, substituting (\ref{eqn:Frame_series}) into (\ref{eqn:Frame_inv_series}) and matching the coefficients of $\rho^{q}$ shows that $b_{0}=a_{0}^{-q/\nu}$. In particular, since $a_{0}>0$ we also have $b_{0}>0$. Choose a circle $C$ in the complex $\rho$-plane centered at $0$ so that the image $C^{\prime}$ under the map $\rho\mapsto u(\rho)$ is a simple curve which winds around $0$ once in the complex $u$-plane. Choose $C$ to be small enough that $\left\vert\sum_{j=1}^{\infty}a_{j}\rho^{j}\right\vert <a_{0}$ for $\rho\in C$, $\left\vert \sum_{j=1}^{\infty}b_{j}u^{j}\right\vert \,<b_{0}$ for $u\in C^{\prime},$ and there are no other zeroes of $\sum_{j=1}^{\infty}a_{j}\rho^{j}$ inside $C$ or zeroes of $\sum_{j=1}^{\infty}b_{j}u^{j}$ inside $C^{\prime}.$

We can now use Cauchy's integral formula\footnote{Recall that $\int_{C}z^{k-1}dz=2\pi i$ if $k=0$ and $=0$ if $k$ is any nonzero integer.} to extract the desired coefficients $b_{j}$:
\begin{align*}
b_{j}  &  =\frac{1}{2\pi i}\int_{C^{\prime}}\left(  \frac{\rho}{u}\right)^{q}u^{-(j+1)}\,du\\
&  =\frac{1}{2\pi i}\int_{C}\frac{\rho^{q}}{u^{q+j+1}}\frac{du}{d\rho}d\rho\\
&  =\frac{1}{2\pi i}\int_{C}\rho^{q}\frac{d}{d\rho}\left(  -\frac{u^{-q-j}}{q+j}\right)  d\rho\\
&  =\frac{1}{2\pi i}\int_{C}\frac{u^{-q-j}}{q+j}q\rho^{q-1}d\rho
\end{align*}
where we made a change of variables in the first line, the second line is valid because we assume $q>0$, and the last line is integration by parts.

Next, $u=\rho(a_{0}+\alpha(\rho))^{1/\nu}$, where $\alpha(\rho):=\sum_{j=1}^{\infty}a_{j}\rho^{j}$, so we have
\begin{align*}
b_{j}  &  =\frac{1}{2\pi i}\int_{C}\frac{q}{q+j}\rho^{-q-j}(a_{0}+\alpha(\rho))^{-(q+j)/\nu}\rho^{q-1}d\rho\\
&  =\frac{1}{2\pi i}\int_{C}\frac{q}{q+j}\rho^{-j-1}\sum_{r=0}^{\infty}\binom{-\frac{q+j}{\nu}}{r}a_{0}^{-(q+j)/\nu-r}\alpha(\rho)^{r}d\rho\\
&  =\frac{q}{q+j}\sum_{r=0}^{\infty}\binom{-\frac{q+j}{\nu}}{r}a_{0}^{-(q+j)/\nu-r}\left\{  \frac{1}{2\pi i}\int_{C}\alpha(\rho)^{r}\rho^{-j-1}d\rho\right\}  .
\end{align*}
By Cauchy's integral formula, $\frac{1}{2\pi i}\int_{C}\alpha^{r}\rho^{-j-1}d\rho=a_{j}^{(r)}$ is the sum of all ordered products of $r$ elements of the set $\{a_{1},a_{2},\dots\}$ whose subscripts add to $j$. This sum is only nonzero for $1\leq r\leq j$, so we obtain
\[
b_{j}=\frac{q}{j+q}\,a_{0}^{-(j+q)/\nu}\sum_{r=1}^{j}\binom{-\frac{j+q}{\nu}}{r}a_{0}^{-r}a_{j}^{(r)}.
\]
\end{proof}

\bigskip

The results of Fulks and Sather from which Theorem \ref{thm:main} is built are summarized in the following lemma\footnote{Fulks and Sather use this to show that an asymptotic expansion of the form (\ref{eqn:asympexp}) exists, and they compute the first term $\zeta_{0}$.} (for completeness, we include a proof which outlines their argument). Define
\[
\widetilde{f}:=\rho^{\nu}\sum_{j=0}^{N}f_{j}(\Omega)\rho^{j}
\]
and
\[
\widetilde{R}_{t}:=\{\mathbf{x}:\widetilde{f}(x)<t\}.
\]

\begin{lemma}
\label{lemma:Fulks-Sather}\cite{Fulks-Sather} Under the same hypotheses as Theorem \ref{thm:main}, for every sufficiently small $a>0$
\[
\int_{R}e^{-kf}g~d^{d}\mathbf{x}=k\int_{0}^{a}e^{-kt}\widetilde{G}(t)dt+o(k^{-(N+\lambda)/\nu}),~k\rightarrow\infty
\]
where
\[
\widetilde{G}(t):=\int_{\widetilde{R}_{t}}g\,d^{d}\mathbf{x}.
\]
\end{lemma}

\begin{remark}
As noted in the remarks following Theorem \ref{thm:main}, if we make the slightly stronger assumption (\ref{eqn:ghyp}), then we can conclude that
\[
\int_{R}e^{-kf}g~d^{d}\mathbf{x}=k\int_{0}^{a}e^{-kt}\widetilde{G}(t)dt+O(k^{-(N+\lambda+min\{1,\beta\})/\nu}),~k\rightarrow\infty
\]
\end{remark}

\begin{proof}
Fulks and Sather \cite[p188--189]{Fulks-Sather} show that we may assume without loss of generality that $g\geq0$. Also, if $\int_{R}e^{-kf} g\,d^{d}\mathbf{x}$ converges for $k_{0}>0$, then it converges for all $k>k_{0}$ since $e^{-k_{0}f}g$ is a dominating function for the integrand. We henceforth assume $k>k_{0}.$

Given $\varepsilon>0$, define the functions $f_{+}(\rho,\Omega)$ and $f_{-}(\rho,\Omega)$ by
\[
f_{\pm}(\rho,\Omega):=\rho^{\nu}\sum_{j=0}^{N}f_{j}(\Omega)\rho^{j}\pm\varepsilon\rho^{N+\nu},
\]
and let
\[
R_{t}^{\pm}:=\{\mathbf{x}:f_{\pm}(\mathbf{x})\leq t\}.
\]
Choose $\rho_{0}>0$ so small that for $0<\rho<\rho_{0}$

\begin{t_enumerate}
\item $\left\vert f(\rho,\Omega)-\rho^{\nu}\sum_{j=0}^{N}f_{j}(\Omega)\rho^{j}\right\vert <\varepsilon\rho^{N+\nu},$

\item $\left\vert g(\rho,\Omega)-\rho^{\lambda-d}\sum_{j=0}^{N}g_{j}(\Omega)\rho^{j}\right\vert \,<\varepsilon\rho^{N+\lambda-d},$

\item $f_{\pm}(\rho,\Omega)$ is increasing in $\rho$ for each $\Omega$, and

\item $R_{0}:=\{\rho\leq\rho_{0}\}\subset R$.
\end{t_enumerate}

The first two conditions are obtainable from the assumed asymptotic series for $f$ and $g$ as $\rho\rightarrow0$, the third condition is insured by the assumption that $f_{0}>0$ away from $\mathbf{0}$, and the last condition follows from the assumption that $\mathbf{0}$ is a point in the interior of $R$.

Then
\[
I(k):=\int_{R}e^{-kf}g\,d^{d}\mathbf{x}%
\]
can be written as
\begin{equation}
I(k)=\int_{R_{0}}e^{-kf}g\,d^{d}\mathbf{x}+\int_{R\setminus R_{0}}e^{-kf}g\,d^{d}\mathbf{x}\mathbf{.} \label{eqn:FSlemma-int1}
\end{equation}
Since we assume that, away from $\mathbf{0,}$ the function $f$ is bounded away from $0$, there exists $A>0$ such that $f\geq A$ for $\rho\geq\rho_{0}$, which implies the second integral above is $O(e^{-kA})$ as $k\rightarrow\infty$.

Next, define
\[
I_{\pm}(k):=\int_{R_{0}}e^{-kf_{\pm}}g\,d^{d}\mathbf{x}\mathbf{,}
\]
and
\[
G_{\pm}(t):=\int_{R_{t}^{\pm}}g\,d^{d}\mathbf{x}.
\]
Since $f_{-}\leq\widetilde{f}\leq f_{+}$, we have
\begin{equation}
R_{t}^{+}\subset\widetilde{R}_{t}\subset R_{t}^{-}. \label{eqn:FSlemmaRs}
\end{equation}
Choose $a>0$ so small that $R_{a}^{-}\subset R_{0}.$ Combining \cite[Lemma 2]{Fulks-Sather} with the same argument we used above to conclude that the second integral in (\ref{eqn:FSlemma-int1}) is exponentially small, we obtain
\begin{align}
I_{\pm}(k)  &  =\int_{R_{a}^{\pm}}e^{-kf_{\pm}}g\,d^{d}\mathbf{x}+\int_{R_{0}\setminus R_{a}^{\pm}}e^{-kf_{\pm}}g\,d^{d}\mathbf{x}\nonumber\\
&  \mathbf{=}\int_{0}^{a}e^{-kt}dG_{\pm}(t)+O(e^{-kA_{\pm}})\label{eqn:FSlemma-Ipmestimate}\\
&  =k\int_{0}^{a}e^{-kt}G_{\pm}(t)dt+O(e^{-kA_{\pm}})\text{ \qquad(integration by parts)}\nonumber
\end{align}
for some positive constants $A_{\pm}.$

By (\ref{eqn:FSlemmaRs}) we have
\begin{equation}
\int_{0}^{a}e^{-kt}G_{+}(t)dt\leq\int_{0}^{a}e^{-kt}\widetilde{G}(t)dt\leq \int_{0}^{a}e^{-kt}G_{-}(t)dt. \label{eqn:FSlemma-Gineq}
\end{equation}
We also have (by the definition of $f_{\pm}$)
\[
I_{+}(k)\leq\int_{R_{o}}e^{-kf}g\,d^{d}\mathbf{x}\leq I_{-}(k)
\]
which, by (\ref{eqn:FSlemma-Ipmestimate}), and (\ref{eqn:FSlemma-int1}) and
the subsequent argument, gives
\[
k\int_{0}^{a}e^{-kt}G_{+}(t)dt+O(e^{-kA_{+}})\leq I(k)+O(e^{-kA})\leq k\int_{0}^{a}e^{-kt}G_{-}(t)dt+O(e^{-kA_{-}}).
\]
Subtracting (\ref{eqn:FSlemma-Gineq}) then yields
\begin{align}
k\int_{0}^{a}e^{-kt}(G_{+}(t)-G_{-}(t))dt+O(e^{-kA_{+}})  &  \leq I(k)-k\int_{0}^{a}e^{-kt}\widetilde{G}(t)dt+O(e^{-kA})\label{eqn:FSlemma-bigineq}\\
&  \qquad\qquad\leq k\int_{0}^{a}e^{-kt}(G_{-}(t)-G_{+}(t))dt+O(e^{-kA_{-}}).\nonumber
\end{align}
Fulks and Sather \cite[p191]{Fulks-Sather} show that there is a constant\footnote{In fact, $c=\frac{2}{\nu}\Gamma\left(  \frac{N+\lambda}{\nu}\right)  \int_{S^{d-1}}g_{0}(\Omega)f_{0}(\Omega)^{-1-(N+\lambda)/\nu} d\Omega$.} $c>0$ such that
\[
k\int_{0}^{a}e^{-kt}(G_{-}(t)-G_{+}(t))dt=\varepsilon ck^{-(N+\lambda)/\nu}+o(k^{-(N+\lambda)/\nu}).
\]
Equation (\ref{eqn:FSlemma-bigineq}) thus becomes
\[
\left\vert I(k)-k\int_{0}^{a}e^{-kt}\widetilde{G}(t)dt\right\vert \leq\varepsilon ck^{-(N+\lambda)/\nu}+o(k^{-(N+\lambda)/\nu}).
\]
Hence we conclude that for every $\varepsilon>0$,
\[
\lim_{k\rightarrow\infty}\left\vert I(k)-k\int_{0}^{a}e^{-kt}\widetilde{G}(t)dt\right\vert k^{(N+\lambda)/\nu}\leq\varepsilon c,
\]
from which the lemma follows.
\end{proof}

\bigskip

Our proof of Theorem \ref{thm:main} is now essentially an estimate of the function $\widetilde{G}(t)$ using the inversion formula, Theorem \ref{thm:Frame}, of Frame.

\bigskip

\begin{proof}
[Proof of Theorem \ref{thm:main}]
As in the proof of Lemma \ref{lemma:Fulks-Sather}, we henceforth assume that $g$ is positive and that $k>k_{0}$ so $\int_{R}e^{-kf}g\,d^{d}\mathbf{x}$ converges. By Lemma \ref{lemma:Fulks-Sather}, it is enough to show that for $a>0$ sufficiently small, there exists an asymptotic expansion
\[
k\int_{0}^{a}e^{-kt}\widetilde{G}(t)dt=\sum_{j=0}^{N}\zeta_{j}k^{-(j+\lambda)/\nu}+o(k^{-(N+\lambda)/\nu}),~k\rightarrow\infty
\]
where the coefficients $\zeta_{j}$ are given by (\ref{eqn:gencoeffs}). To this end, choose $\rho_{0}$ small enough that $\widetilde{f}$ is increasing in $\rho$ for $0\leq\rho\leq\rho_{0},$ and then choose $a$ small enough that $\widetilde{R}_{a}:=\{\mathbf{x}:\widetilde{f}(\mathbf{x})\leq a\}\subset \{\rho\leq\rho_{0}\}$. Then for each $t$ with $0\leq t\leq a$, the equation $t=\widetilde{f}(\rho,\Omega)$ has a unique solution $\rho(t,\Omega)$ for $\rho$ which is continuous in $\Omega$. Substituting the series (\ref{eqn:g-series}) for $g$ and performing the integration over $\rho$ yields
\begin{equation}
\widetilde{G}(t)=\int_{S^{d-1}}\int_{0}^{\rho(t,\Omega)}g(\rho,\Omega)\rho^{d-1}d\rho\,d\Omega=\int_{S^{d-1}}\left[  \sum_{j=0}^{N}\frac{g_{j}(\Omega)}{j+\lambda}\rho^{j+\lambda}(t\,,\Omega)+o(\rho^{N+\lambda})\right]  d\Omega\label{eqn:Gtilde}
\end{equation}

Now we use Theorem \ref{thm:Frame} to explicitly estimate the powers of the inverse function $\rho(t,\Omega)$, thus obtaining explicit estimates for $\widetilde{G}(t)$.

Let
\[
a_{j}:=
\begin{cases}
f_{j}(\Omega) & \text{ for }0\leq j\leq N,\text{ and }\\
0 & \text{ for}j>N.
\end{cases}
\]
Then $t:=\widetilde{f}=\rho^{\nu}\sum_{j=0}^{\infty}a_{j}\rho^{j}$.

Let us briefly consider the inverse function of $t=\widetilde{f}(\rho,\Omega),$ for fixed $\Omega$, in more detail. Let $u=t^{1/\nu}$, so that $u=\rho\left(  \sum_{j=0}^{\infty}a_{j}\rho^{j}\right)  ^{1/\nu}.$ One may compute that $\left.  \frac{d}{d\rho}u\right\vert _{\rho=0}=a_{0}^{1/\nu}=f_{0}(\Omega)^{1/\nu}$, which is positive by assumption. Moreover, $u$ is an analytic function of $\rho$ in a neighborhood of $0$. Since $u(\rho=0)=0$, the inverse function is also analytic in a neighborhood of $0$, and we may express it as a series $\rho=u\sum_{j=0}^{\infty}b_{j}u^{j}.$ Taking positive powers of this series yields series $\rho^{q}=u^{q}\sum_{j=0}^{\infty}b_{j}^{(q)}u^{j}.$ In particular, the two series $u^{\nu}=\rho^{\nu}\sum_{j=0}^{\infty}a_{j}\rho^{j}$ and $\rho^{q}=u^{q}\sum_{j=0}^{\infty}b_{j}^{(q)}u^{j}$ satisfy the hypotheses of Theorem \ref{thm:Frame}.

Hence, with $u=t^{1/\nu},$ Theorem \ref{thm:Frame} yields $\rho^{q}=u^{q}\sum_{j-0}^{\infty}b_{j}^{(q)}u^{j},$ where the coefficients $b_{j}^{(q)}$ are given by (\ref{eqn:inv_coeffs}). Clearly, $b_{j}^{(q)}$ depends only on $\{a_{0},a_{1},\dots,a_{j}\}$. \ Moreover, the remainder after $N$ terms, $u^{q}\sum_{N+1}^{\infty}b_{j}^{(q)}(\Omega)u^{j},$ is uniformly bounded for $\Omega\in S^{d-1}$ and $0\leq u\leq a^{1/\nu};$ that is,
\[
u^{q}\sum_{j=N+1}^{\infty}b_{j}^{(q)}(\Omega)u^{j}=O(u^{N+1+q})=o(u^{N+q}).
\]

Since $u=t^{1/\nu},$ we get
\begin{equation}
\rho^{q}=t^{q/\nu}\sum_{j=0}^{N}b_{j}^{(q)}(\Omega)t^{j/\nu}+o(t^{(N+q)/\nu}).
\label{eqn:rho}
\end{equation}
Substituting (\ref{eqn:rho}) into the expression (\ref{eqn:Gtilde}) for $\widetilde{G}(t)$, we obtain
\begin{equation}
\widetilde{G}(t)=\sum_{j=0}^{N}\sum_{l=0}^{N}t^{\frac{l+j+\lambda}{\nu}} \left[\frac{1}{j+\lambda}\int_{S^{d-1}}g_{j}(\Omega)b_{l}^{(j+\lambda)}(\Omega)d\Omega\right]  +o(t^{\frac{N+\lambda}{\nu}}). \label{eqn:G_long}
\end{equation}
After some rearrangement,\footnote{To obtain (\ref{eqn:G_short}), define $h_{j,l}:=\frac{1}{j+\lambda}\int_{S^{d-1}}g_{j}(\Omega)b_{l}^{(j+\lambda)}(\Omega)d\Omega.$ Then the double sum in (\ref{eqn:G_long}) is $\sum_{j=0}^{N}\sum_{l=0}^{N}h_{j,l}t^{\frac{j+l+\lambda}{\nu}}.$ Of course, if both $j$ and $l$ are large, then the corresponding term is absorbed into the error term. So in fact we only need consider
\[
\sum_{j=0}^{N}\sum_{l+j\leq N}h_{j,l}t^{\frac{j+l+\lambda}{\nu}} =\sum_{j=0}^{N}\sum_{l=0}^{N-j}h_{j,l}t^{\frac{j+l+\lambda}{\nu}}=\sum_{m=0}^{N}\sum_{l+j=m}h_{j,l}t^{\frac{m+\lambda}{\nu}} =\sum_{m=0}^{N}\sum_{l=0}^{m}h_{m-l,l}t^{\frac{m+\lambda}{\nu}}.
\]
Putting $h_{m-l,l}$ into this expression then yields (\ref{eqn:G_short}).} one obtains
\begin{equation}
\widetilde{G}(t)=\sum_{j=0}^{N}\eta_{j}t^{\frac{j+\lambda}{\nu}}+o(t^{\frac{N+\lambda}{\nu}}) \label{eqn:G_short}
\end{equation}
where
\[
\eta_{j}=\sum_{l=0}^{j}\frac{1}{j-l+\lambda}\int_{S^{d-1}}g_{j-l}(\Omega)b_{l}^{(j-l+\lambda)}(\Omega)d\Omega.
\]

Finally, recall that for $s>-1$
\[
k\int_{0}^{a}e^{-kt}t^{s}dt=k^{-s}\Gamma(s+1)+o(k^{-\infty}).
\]
We multiply $\widetilde{G}(t)$ by $e^{-kt}$ and integrate termwise to get
\begin{align*}
k\int_{0}^{a}e^{-kt}\widetilde{G}(t)dt  &  =k\int_{0}^{a}e^{-kt}\left(\sum_{j=0}^{N}\eta_{j}t^{\frac{j+\lambda}{\nu}}+o(t^{\frac{N+\lambda}{\nu}})\right)  dt+O(k^{-\infty})\\
&  =\sum_{j=0}^{N}\zeta_{j}k^{-(j+\lambda)/\nu}+O(k^{-(N+\lambda)/\nu})
\end{align*}
where, since $\Gamma(s+1)=s\Gamma(s),$
\[
\zeta_{j}=\tfrac{1}{\nu}\Gamma(\tfrac{j+\lambda}{\nu})\int_{S^{d-1}}\left[f_{0}(\Omega)^{-(j+\lambda)/\nu} \sum_{m=0}^{j}g_{j-m}(\Omega)\sum_{r=1}^{m}\binom{-\frac{j+\lambda}{\nu}}{r}\frac{f_{m}^{(r)}(\Omega)}{f_{0}(\Omega)^{r}}\right]  d\Omega.
\]
\end{proof}

\bigskip

We conclude this section with a short discussion of the combinatorial quantities which arise in Theorem \ref{thm:main}. The quantities
\begin{equation}
f_{m}^{(r)}=\sum_{\substack{n_{1}+n_{2}+\cdots+n_{r}=m\\n_{i}\geq1,~i=1,\dots,r}}f_{n_{1}}f_{n_{2}}\cdots f_{n_{r}} \label{eqn:fmr}
\end{equation}
also arise in the $1$-dimensional case studied by Wojdylo in \cite{WojdyloSIAM}, \cite{WojdyloJCAM}, where he calls them ``partial ordinary Bell polynomials.'' They are relatively simple combinatorial objects to compute; for example, $f_{n}^{(r)}$ may be computed as the coefficient of $x^{n}$ appearing in $(f_{1}x+f_{2}x^{2}+f_{3}x^{3}+\cdots)^{r}$. The following result describes a simple recursive algorithm for their computation. Explicit algorithms suitable for computer implementation may be found in \cite{WojdyloSIAM}, \cite{WojdyloJCAM}.

\begin{proposition}
For $n>0$ we have $f_{n}^{(1)}=f_{n}$, and for $1<r\leq n,$ the coefficients $f_{n}^{(r)}$ can be computed recursively via $\displaystyle{f_{n}^{(r)}=\sum_{j=r-1}^{n-1}f_{n-j}f_{j}^{(r-1)}.}$
\end{proposition}

%%%%%%%%%%%%%%%%%%%%%%%%%%%%%%%%%%%%%%%%%%%%%%%%%%%%%%%%%%%%%%%%%%
\section{Further Simplifications\label{sec:corollaries}}
%%%%%%%%%%%%%%%%%%%%%%%%%%%%%%%%%%%%%%%%%%%%%%%%%%%%%%%%%%%%%%%%%%

In this section, we make additional hypotheses which allow one to evaluate the integrals which appear in Theorem \ref{thm:main}. In Section \ref{sec:f0consttaylor}, we assume first that $f_0(\Omega)$ is constant, that is, that $f/\rho^2$ is continuous at the origin, which allows us to make slight simplifications to the coefficients $\xi_j.$ We then assume that $f$ and $g$ are smooth enough to admit at least finite order Taylor series expansions, and evaluate the integrals appearing in the coefficients $\xi_j$ of Theorem \ref{thm:main} explicitly, thus yielding an asymptotic expansion of $\int_R e^{-kf}g\,d^d\mathbf{x}$ in terms of the Taylor coefficients. In Section \ref{sec:f0constv=2}, we make the assumption that $f$ attains a nondegenerate minimum at $\mathbf{0}$ (this is what is traditionally assumed to make Laplace's approximation) and also that $f$ and $g$ are smooth enough to admit at least finite order Taylor series expansions. These hypotheses also allow us to evaluate the integrals appearing in the coefficients $\xi_j$ explicitly, and lead us to a proof of Theorem \ref{thm:best}, which was stated in the introduction.

\subsection{Simplifications with $f_{0}(\Omega)$ constant, and with enough smoothness for Taylor
series.\label{sec:f0consttaylor}}

First, we simplify Theorem \ref{thm:main} under the assumption that $f_{0}(\Omega)$ is constant. Then, we show that if $f_{0}(\Omega)$ is constant and $f$ and $g$ are given as Taylor series with remainder, the integrals appearing in the coefficients $\zeta_{j}$ in Theorem \ref{thm:main} can be explicitly computed to obtain a closed formula for $\zeta_{j}$ involving the coefficients of the Taylor series.

\bigskip

If $f_{0}(\Omega)=f_{0}$ is constant, then the coefficients of Theorem \ref{thm:main} simplify significantly since the denominators consist entirely of terms involving $f_{0}(\Omega)$. In general, one does not expect $f_{0}(\Omega)$ to be constant, but in the common case that $\nu=2$ (discussed in \ref{sec:f0constv=2} below), an appropriate coordinate change guarantees $f_{0}(\Omega)=1/2.$

Before stating the main asymptotic expansion with the additional hypothesis that $f_{0}$ is constant, we mention an easy lemma characterizing the condition that $f_{0}$ is constant.

\begin{lemma}
Suppose $f$ is a measurable function on a measurable set $R\subset \mathbb{R}^{d}$ containing $\mathbf{0}$ as an interior point, such that for some positive integer $N$ and some positive real number $\nu$
\[
f(\rho,\Omega)=\rho^{\nu}\sum_{j=0}^{N}f_{j}(\Omega)\rho^{j}+o(\rho^{N+\nu}).
\]
Then $\rho^{-\nu}f(\rho,\Omega)$ is continuous at $\rho=0$ if and only if $f_{0}(\Omega)$ is constant.
\end{lemma}

\begin{proof}
The function $\rho^{-\nu}f(\rho,\Omega)$ is continuous at $\rho=0$ if and only if $c:=\lim_{\rho\rightarrow0}\rho^{-\nu}f(\rho,\Omega)=f_{0}(\Omega)$ is constant.
\end{proof}

\begin{corollary}
\label{cor:f0const}
If, in addition to the hypotheses of Theorem \ref{thm:main}, $\rho^{-\nu}f(\rho,\Omega)$ is continuous at $\rho=0$, so that $f_{0}(\Omega)=f_{0}$ is constant, then there exists an asymptotic expansion
\begin{equation}
\int_{R}e^{-kf}g\,d^{d}\mathbf{x}=\sum_{j=0}^{N}\zeta_{j}k^{-(j+\lambda)/\nu}+o(k^{-(N+\lambda)/\nu}),~k\rightarrow\infty
\end{equation}
where the coefficients are given by
\[
\zeta_{j}=\frac{1}{\nu}\Gamma\left(  \tfrac{j+\lambda}{\nu}\right)f_{0}^{-(j+\lambda)/\nu}\sum_{m=0}^{j}\sum_{r=1}^{m} \binom{-\frac{j+\lambda}{\nu}}{r}f_{0}^{-r}\int_{S^{d-1}}g_{j-m}(\Omega)f_{m}^{(r)}(\Omega)d\Omega.
\]
Empty sums are understood to be $1$.
\end{corollary}

\bigskip

Now we suppose further that $f$ and $g$ are smooth enough to be given as Taylor series (with remainder), and give closed form expressions for the coefficients appearing in Theorem \ref{thm:main} (still with the assumption that $f_{0}(\Omega)$ is constant); that is, we explicitly evaluate the spherical integrals appearing in the coefficients $\zeta_{j}$ of Theorem \ref{thm:main} to obtain formulas in terms of the Taylor coefficients.

To relate the Taylor series expansions with the radial expansions of Theorem \ref{thm:main}, we introduce spherical coordinates $(\rho,\phi_{1},\phi_{2},\dots,\phi_{d-1})$ on $\mathbb{R}^{d}$; these are related to the Cartesian coordinates $\mathbf{x}=(x^{1},\dots,x^{d})$ via
\begin{align}
x^{1}  &  =\rho\cos\phi_{1},\ \nonumber\\
x^{2}  &  =\rho\sin\phi_{1}\cos\phi_{2},\nonumber\\
&  \vdots\label{eqn:sphericalcoords}\\
x^{d-1}  &  =\rho\sin\phi_{1}\sin\phi_{2}\cdots\sin\phi_{d-2}\cos\phi_{d-1},\nonumber\\
x^{d}  &  =\rho\sin\phi_{1}\sin\phi_{2}\cdots\sin\phi_{d-2}\sin\phi_{d-1}.\nonumber
\end{align}
The coordinate ranges are $0\leq\rho<\infty,$ $0\leq\phi_{j}<\pi$ for $j=1,\dots,d-2$ and $0\leq\phi_{d-1}<2\pi$. The solid angle element in spherical coordinates is
\begin{equation}
d\Omega=\sin^{d-2}\phi_{1}\sin^{d-3}\phi_{2}\cdots\sin\phi_{d-2}d\phi_{1}\cdots d\phi_{d-1}. \label{eqn:solidangle}
\end{equation}
According to (\ref{eqn:sphericalcoords}), the unit vector $\Omega=\mathbf{x}/\rho$ may be expressed as
\[
\Omega=(\cos\phi_{1},\sin\phi_{1}\cos\phi_{2},\dots,\sin\phi_{1}\sin\phi_{2}\cdots\sin\phi_{d-1}).
\]

Suppose $f$ and $g$ are given as Taylor series with remainder, in standard multi-index notation, as
\begin{equation}
f(\mathbf{x})=\sum_{\left\vert \alpha\right\vert =\nu}^{N+\nu}\frac{1}{\alpha!}D^{\alpha}f(\mathbf{0})\mathbf{x}^{\alpha}+R_{N+\nu+1}(\mathbf{x}),
\label{eqn:Taylorf1}
\end{equation}
for some integer $\nu>0$ (actually, if $f$ is to have a unique minimum of $0$ at $\mathbf{0}$, the parameter $\nu$ must be an even integer $\geq2$), and
\begin{equation}
g(\mathbf{x})=\sum_{\left\vert \beta\right\vert =0}^{N}\frac{1}{\beta!}D^{\beta}g(\mathbf{0})\mathbf{x}^{\beta}+R_{N+1}^{\prime}(\mathbf{x}).
\label{eqn:Taylorg1}
\end{equation}
Observe that (\ref{eqn:Taylorg1}) implies $\lambda=d$. To compare these expressions with the radial power series used in Theorem \ref{thm:main}, we write (\ref{eqn:Taylorf1}) and (\ref{eqn:Taylorg1}) in spherical coordinates to obtain
\begin{align*}
f(\rho,\Omega)  &  =\sum_{j=\nu}^{N+\nu}\rho^{j}\sum_{\left\vert\alpha\right\vert =j}\frac{D^{\alpha}f(\mathbf{0})}{\alpha!}\Omega^{\alpha}+O(\rho^{N+\nu+1})\text{ and}\\
g(\rho,\Omega)  &  =\sum_{j=0}^{N}\rho^{j}\sum_{\left\vert \alpha\right\vert=j}\frac{D^{\alpha}g(\mathbf{0})}{\alpha!}\Omega^{\alpha}+O(\rho^{N+1})
\end{align*}
from which we see that
\begin{equation}
f_{j}(\Omega)=\sum_{\left\vert \alpha\right\vert =j+\nu}\frac{D^{\alpha}f(\mathbf{0})}{\alpha!}\Omega^{\alpha}\text{, and }g_{j}(\Omega
)=\sum_{\left\vert \alpha\right\vert =j}\frac{D^{\alpha}g(\mathbf{0})}{\alpha!}\Omega^{\alpha}. \label{eqn:fjgjTaylor}
\end{equation}

\begin{corollary}
\label{cor:taylor}If, in addition to the hypotheses of Theorem \ref{thm:main}, the functions $f$ and $g$ are smooth enough to be given as the Taylor series (\ref{eqn:Taylorf1}) and (\ref{eqn:Taylorg1}) where $f_{0}(\Omega)=f_{0}$ is constant, then there exists an asymptotic expansion
\[
\int_{R}e^{-kf}g\,d^{d}\mathbf{x}=k^{-d/\nu}\sum_{j=0}^{N}\zeta_{j}k^{-j/\nu}+O(k^{-(N+d+1)/\nu})
\]
where the coefficients are given by
\begin{align*}
\zeta_{j} &  =\tfrac{1}{\nu}\Gamma\left(  \tfrac{d+j}{\nu}\right) f_{0}^{-(j+d)/\nu}\sum_{m=0}^{j}\sum_{r=1}^{m}\binom{-\frac{d+j}{\nu}}{r}f_{0}^{-r}\\
&  \qquad\times\sum_{\left\vert \beta\right\vert =j-m}\sum_{\left\{\substack{n_{1}+n_{2}+\cdots+n_{r}=m\\n_{i}\geq1,~i=1,\dots,r}\right\}  } \sum_{\left\vert \alpha_{1}\right\vert =n_{1}+\nu}\cdots\sum_{\left\vert\alpha_{r}\right\vert =n_{r}+\nu}w_{\beta+\alpha_{1}+\cdots+\alpha_{r}} \frac{D^{\beta}g(\mathbf{0})D^{\alpha_{1}}f(\mathbf{0})\cdots D^{\alpha_{r}}f(\mathbf{0})}{\beta!\alpha_{1}!\cdots\alpha_{r}!}
\end{align*}
where for a multi-index $\alpha=(\alpha^{1},\alpha^{2},\dots,\alpha^{d})$,
\begin{equation}
w_{\alpha}:=(2\pi)^{d/2}\left\{
\begin{array}
[c]{c}%
1\text{ if }d\text{ even}\\
\sqrt{2/\pi}\text{ if }d\text{ odd}
\end{array}
\right\}  \frac{(\alpha-\mathbf{1})!!}{(\left\vert \alpha\right\vert
+d-2)!!}\operatorname*{even}(\alpha).\label{eqn:w}
\end{equation}
Empty sums are understood to be $1$.
\end{corollary}

Recall that for an integer $n$, the double factorial is defined\footnote{We \textit{could} define $w_{\alpha}$ by the nonzero formula appearing in (\ref{eqn:w}) if we adopted the nonstandard convention that for $n$ even,
\[
n!!=n(n-2)\cdots4\cdot2\cdot0=0.
\]
We will not do this.} to be $n!!:=n(n-2)(n-4)\cdots(n-2\lfloor n/2\rfloor+2)$ where $\lfloor n/2\rfloor$ is the greatest integer less than $n/2$, that is,
\[
n!!=
\begin{cases}
n(n-2)\cdots4\cdot2,\text{ if $n$ even}\\
n(n-2)\cdots5\cdot3,\text{ if $n$ odd}
\end{cases},
\]
and we adopt the convention that $(-1)!!=1.$ Moreover, for a multi-index $\alpha$, we define $a!!:=\alpha^{1}!!\alpha^{2}!!\cdots\alpha^{d}!!.$ Finally, $\operatorname*{even}(\alpha):=1$ if all components of $\alpha$ are even, and $0$ otherwise.

\bigskip

\begin{proof}
By Corollary \ref{cor:f0const} and Remark 1 following Theorem \ref{thm:main}, we need only compute the integrals $\int_{S^{d-1}}g_{j-m}f_{m}^{(r)}d\Omega.$ The coefficients $f_{j}(\Omega)$ and $g_{j}(\Omega)$ obtained from the Taylor series (\ref{eqn:fjgjTaylor}) and (\ref{eqn:Taylorg1}) are given by (\ref{eqn:fjgjTaylor}). From (\ref{eqn:fmr}), the integrand of interest is therefore given by
\begin{equation}
g_{j-m}(\Omega)f_{m}^{(r)}(\Omega)=\sum_{\left\vert \beta\right\vert =j-m} \sum_{\left\{  \substack{\sum_{i=1}^{r}n_{i}=m\\n_{i}\geq1,~i=1,\dots ,r}\right\}  }\sum_{\left\vert \alpha_{1}\right\vert =n_{1}+\nu}\cdots \sum_{\left\vert \alpha_{r}\right\vert =n_{r}+\nu}\frac{D^{\beta} g(\mathbf{0})D^{\alpha_{1}}f(\mathbf{0})\cdots D^{\alpha_{r}}f(\mathbf{0})}{\beta!\alpha_{1}!\cdots\alpha_{r}!}\Omega^{\beta+\alpha_{1}+\cdots+\alpha_{r}}. \label{eqn:longprod}
\end{equation}
Thus, it remains to compute the integral
\[
\int_{S^{d-1}}\Omega^{\alpha}\,d\Omega
\]
for an arbitrary multi-index $\alpha$.

Let $\alpha=(\alpha^{1},\dots,\alpha^{d})$. By (\ref{eqn:sphericalcoords}) and (\ref{eqn:solidangle}), we have
\begin{align}
\Omega^{\alpha}d\Omega &  =\left(  \cos^{\alpha^{1}}\varphi_{1}\sin^{\alpha^{2}+\alpha^{3}+\cdots+\alpha^{d}+d-2}\varphi_{1}\right) \times\left(  \cos^{\alpha^{2}}\varphi_{2}\sin^{\alpha^{3}+\alpha^{4}+\cdots+\alpha^{d}+d-3}\right)  \times\cdots\label{eqn:OmegaAlpha}\\
&  \qquad\times\left(  \cos^{\alpha^{d-1}}\varphi_{d-1}\sin^{\alpha^{d}}\varphi_{d-1}\right)  d\varphi_{1}\cdots d\varphi_{d-1}.
\end{align}
For nonnegative integers $a$ and $b$, an elementary computation (using reduction formulas, for example \cite[2.510]{Gradshteyn-Ryzhik-5th}) shows that \[
\int_{0}^{\pi}\sin^{a}x\cos^{b}x\,dx=\frac{(a-1)!!(b-1)!!}{(a+b)!!}\,c_{1}(a,b),
\]
where
\[
c_{1}(a,b):=
\begin{cases}
0 & \text{ if $b$ is odd,}\\
\pi & \text{ if both }a,b\text{ are even,}\\
2 & \text{if }a\text{ is odd and }b\text{ is even}
\end{cases},
\]
and
\[
\int_{0}^{2\pi}\sin^{a}x\cos^{b}x\,dx=\frac{(a-1)!!(b-1)!!}{(a+b)!!}\,c_{2}(a,b),
\]
where
\[
c_{2}(a,b):=
\begin{cases}
0 & \text{ if }a\text{ or }b\text{ (or both) is odd,}\\
2\pi & \text{ if both }a,b\text{ are even,}
\end{cases}.
\]
Combining this with (\ref{eqn:OmegaAlpha}) shows that
\[
\int_{S^{d-1}}\Omega^{\alpha}d\Omega=\prod_{l=1}^{d-1}\frac{(\alpha^{l}-1)!!\left(  \sum_{p=l+1}^{d}\alpha^{p}+d-2-l\right)  !!}{\left(\sum_{p=l}^{d}\alpha^{p}+d-1-l\right)  !!}\times\prod_{l=1}^{d-2}c_{1}\left(\sum_{p=l+1}^{d}\alpha^{p}+d-1-l,\alpha^{l}\right)  \times c_{2}(\alpha_{d-1},\alpha_{d}).
\]
Now, if any $\alpha^{p}$ is odd, the entire integral is zero, so we assume for the moment that all $\alpha^{p}$ are even, which implies $\sum\alpha^{p}$ is even. If $d$ is even, then $c_{1}\left(  \sum_{p=l+1}^{d}\alpha^{p}+d-1-l,\alpha^{l}\right)  =\pi$ if $l$ is odd and $2$ if $l$ is even. If $d$ is odd, then the opposite is true, that is, $c_{1}\left(  \sum_{p=l+1}^{d}\alpha^{p}+d-1-l,\alpha^{l}\right)  =2$ if $l$ is odd and $\pi$ if $l$ is even. In any case, $c_{2}(\alpha_{d-1},\alpha_{d})=2\pi$. It is now a simple matter to count the factors of $2$ and $\pi$ appearing in the above product in each case of $d$ even or odd, which yields the theorem.
\end{proof}

%%%%%%%%%%%%%%%%%%%%%%%%%%%%%%%%%%%%%%%%%%%%%%%%%%%%%%%%%%%%%%%%%%
\subsection{Assuming a nondegenerate minimum.\label{sec:f0constv=2}}

In this section we consider the case that $f$ has a nondegenerate minimum at $\mathbf{0}$. It follows (from the proof of Theorem \ref{thm:best}, below) that $\nu=2$ and that $f_{0}(\Omega)>0$; in fact, by an appropriate choice of coordinates, we can guarantee that $f_{0}(\Omega)=1/2$. Hence, Corollary \ref{cor:f0const} applies and we obtain Theorem \ref{thm:best}, which was given in the introduction.

\bigskip

\begin{proof}[Proof of Theorem \ref{thm:best}]
Since $f$ has a nondegenerate minimum at $\mathbf{0}$, the Hessian can be diagonalized, $Hf(\mathbf{0})=Q^{-1}DQ,$ for some orthogonal matrix $Q$ and some diagonal matrix $D.$ The eigenvalues of $Hf(\mathbf{0})$ are all positive, so we can unambiguously define $\sqrt{D}$. Let $P:=\sqrt{D}Q$, and define $\mathbf{y}:=P\mathbf{x}$. Denote the Hessian of $f$ with respect to the new coordinates $\mathbf{y}$ by $H_{y}f$. Then $H_{y}f(\mathbf{0})$ is the identity matrix. Moreover, the Jacobian of the coordinate transformation, since $Q$ is orthogonal, is $\det\sqrt{D}^{-1}=\sqrt{H_{x}f(\mathbf{0})}^{-1}.$ Thus, the integral of interest becomes
\[
\int_{R}e^{-kf(\mathbf{x})}g(\mathbf{x)\,}d^{d}\mathbf{x}=\sqrt{\det H_{x}f(\mathbf{0})}^{-1}\int_{PR}e^{-kf(\mathbf{y)}}g(\mathbf{y)}\,d^{d}\mathbf{y.}
\]

Now we apply our results to the last integral above; in particular, we introduce spherical coordinates $\rho=\left\vert \mathbf{y}\right\vert $ and $\Omega=\mathbf{y/\rho}$. Again since $f$ has a nondegenerate minimum at $\mathbf{0}$, the first term in the radial expansion of $f$ at $\rho=0$ is
\[
f_{0}(\Omega)=\frac{1}{2}\sum_{\left\vert \alpha\right\vert =2}D_{y}^{\alpha}f(\mathbf{0})\Omega^{\alpha}=\frac{1}{2}\sum_{j=1}^{d}\Omega_{j}^{2}=\frac{1}{2},
\]
whence we can apply Theorem \ref{cor:f0const}. On the other hand, the smoothness assumptions on $f$ and $g$ imply that we can express them as Taylor series with remainder as
\[
f(\mathbf{y})=\sum_{\left\vert \alpha\right\vert =2}^{N+2}\frac{1}{\alpha !}D_{y}^{\alpha}f(\mathbf{0})\mathbf{y}^{\alpha}+O(\rho^{N+3})\text{ and }g(\mathbf{y})=\sum_{\left\vert \alpha\right\vert \leq N}\frac{1}{\alpha!}D_{y}^{\alpha}g(\mathbf{0})\mathbf{y}^{\alpha}+O(\rho^{N+1}),
\]
and we can apply Corollary \ref{cor:taylor} to obtain
\begin{align*}
\zeta_{j}  &  =\sqrt{\det H_{x}f(\mathbf{0})}^{-1}\tfrac{1}{2}\Gamma\left( \tfrac{d+j}{2}\right)  2^{(j+d)/2}\sum_{m=0}^{j}\sum_{r=1}^{m}\binom {-\frac{d+j}{2}}{r}2^{r}\\
&  \qquad\times\sum_{\left\vert \beta\right\vert =j-m}\sum_{\left\{\substack{n_{1}+n_{2}+\cdots+n_{r}=m\\n_{i}\geq1,~i=1,\dots,r}\right\}  } \sum_{\left\vert \alpha_{1}\right\vert =n_{1}+2}\cdots\sum_{\left\vert\alpha_{r}\right\vert =n_{r}+2}w_{\beta+\alpha_{1}+\cdots+\alpha_{r}} \frac{D^{\beta}_y g(\mathbf{0})D^{\alpha_{1}}_y f(\mathbf{0})\cdots D^{\alpha_{r}}_y f(\mathbf{0})}{\beta!\alpha_{1}!\cdots\alpha_{r}!}.
\end{align*}
Now, if $j$ is odd, then $\left\vert \beta+\alpha_{1}+\cdots+\alpha_{r}\right\vert =j-m+m+2r=j+2r$ is odd, whence $w_{\beta+\alpha_{1}+\cdots+\alpha_{r}}=0$ since at least one component of $\beta+\alpha_{1}+\cdots+\alpha_{r}$ must be odd. This means $\zeta_{2j+1}=0$ and we can replace $j$ by $2j$.

Inserting the value of $w_{\beta+\alpha_{1}+\cdots+\alpha_{r}}$ and using the fact that for each positive integer $j$ we have
\[
\Gamma\left(\tfrac{d}{2}+j\right)=2^{-(j-1)-d/2}(d+2j-2)!!\left\{
\begin{array}
[c]{c}%
1\text{ if }d\text{ is even}\\
\sqrt{\pi/2}\text{ if }d\text{ is odd}%
\end{array}
\right\},
\]
the coefficients $\zeta_{2j},$ after some simple cancelations, become
\begin{multline*}
\zeta_{2j} =\frac{(2\pi)^{d/2}}{\sqrt{\det H_{x}f(\mathbf{0})}} \sum_{m=0}^{2j}\sum_{r=1}^{m}2^{r}\binom{-(\frac{d}{2}+j)}{r}\frac{(d+2j-2)!!}{(d+2j+2r-2)!!}\\
\times\sum_{\left\vert \beta\right\vert =2j-m}\sum_{\left\{\substack{n_{1}+n_{2}+\cdots+n_{r}=m\\n_{i}\geq1,~i=1,\dots,r}\right\}  } \sum_{\left\vert \alpha_{1}\right\vert =n_{1}+2}\cdots\sum_{\left\vert\alpha_{r}\right\vert =n_{r}+2}\frac{D^{\beta}_y g(\mathbf{0)}D^{\alpha_{1}}_y f(\mathbf{0})\cdots D^{\alpha_{r}}_y f(\mathbf{0})}{\beta!\alpha_{1}!\cdots\alpha_{r}!}\\
\times(\beta+\alpha_{1}+\cdots+\alpha_{r})!!\operatorname*{even}(\beta+\alpha_{1}+\cdots+\alpha_{r}).
\end{multline*}
Finally, one may compute that
\[
2^{r}\binom{-(\frac{d}{2}+j)}{r}\frac{(d+2j-2)!!}{(d+2j+2r-2)!!}=\frac{(-1)^{r}}{r!}
\]
to obtain the theorem.
\end{proof}

\bigskip

The coefficients $\zeta_{j}$ of Theorem \ref{thm:best} can be written in terms of the derivatives of $f$ and $g$ with respect to the original variables $\mathbf{x}$ by inserting the appropriate entries of $P^{-1}.$

Finally, we relate Theorem \ref{thm:best} to Laplace's approximation.

\begin{corollary}
(Laplace's Approximation) Let $R\subset\mathbb{R}^{d}$ be a measurable set which contains $\mathbf{0}$ as an interior point. Let $f$ be a twice differentiable function on $R$ which attains a nondegenerate, unique minimum value of $0$ at $\mathbf{0}$, and let $g$ be a continuous function on $R.$ Assume moreover that there exists a positive real number $k_{0}$ such that
\[
\int_{R}e^{-k_{0}f}g~d^{d}\mathbf{x}
\]
converges. Then
\[
\int_{R}e^{-kf}g~d^{d}\mathbf{x}=\left(  \frac{2\pi}{k}\right)  ^{d/2}\frac{g(\mathbf{0})}{\sqrt{\det Hf(\mathbf{0})}}+O(k^{-d/2+1}),~k\rightarrow\infty.
\]

\end{corollary}

\begin{proof}
Since $g$ is continuous, we have $\lambda=d$ and $g_{0}(\Omega)=g(\mathbf{0})$ is constant. The Corollary then follows from the formula for $\zeta_{0}$ of Theorem \ref{thm:best}.
\end{proof}

\bigskip

%%%%%%%%%%%%%%%%%%%%%%%%%%%%%%%%%%%%%%%%%%%%%%%%%%%%%%%%%%%%%%%%%%
\subsection{The $1$-dimensional case.\label{sec:1D}}

We compute the coefficients of Theorem \ref{thm:main} in the $1$-dimensional case, where several simplifications occur, thus rederiving the recent results of Wojdylo \cite{WojdyloSIAM}, \cite[Thm 1.1]{WojdyloJCAM}.

\begin{corollary}
\label{cor:1D}Suppose that for $x\rightarrow0$, we have $f(x)=x^{\nu}\sum_{j=0}^{N}a_{j}x^{j}+o(x^{N+\nu})$ for some $\nu>0$ and $g(x)=\sum_{j=0}^{N}b_{j}x^{j}+o(x^{N})$. Then if there exists $k_{0}>0$ such that $\int_{-a}^{b}e^{-k_{0}f}g~dx$ converges, there exists an asymptotic expansion
\[
\int_{-a}^{b}e^{-kf}g\,dx=k^{-1/\nu}\sum_{j=0}^{\lfloor N/2\rfloor}\zeta_{2j}k^{-2j/\nu}+o(k^{-(2\lfloor N/2\rfloor+1)/\nu}),~k\rightarrow\infty
\]
where the coefficients are given by
\[
\zeta_{2j}=\tfrac{2}{\nu}\,\Gamma\left(  \tfrac{2j+1}{\nu}\right)\,a_{0}^{-(2j+1)/\nu}\left(  \sum_{m=0}^{2j}b_{2j-m}\sum_{r=1}^{m}\binom{-\frac{2j+1}{\nu}}{r}a_{0}^{-r}a_{m}^{(r)}\right)
\]
in which $a_{m}^{(r)}$ is the sum of all ordered products of $r$ terms of the set $\{a_{1},a_{2},\dots\}$ such that the subscripts add to $m$, $\lfloor N/2\rfloor$ denotes the largest integer less than $N/2$ and empty sums are understood to be $1.$
\end{corollary}

\begin{proof}
Note that the hypothesis on $g$ implies $\lambda=d$ in Theorem \ref{thm:main}.

Introduce ``spherical coordinates'' $\rho(x)=\left\vert x\right\vert $ and $\Omega(x)=x/\left\vert x\right\vert=\pm1$ on $\mathbb{R}$. Note that $S^{0}=\{\pm1\}$. Since $\Omega^{j}=1$ if $j$ is even and $=\Omega$ if $j$ is odd,
\begin{equation}
\int_{S^{0}}\Omega^{j}d\Omega=2\varepsilon_{j}
\end{equation}
where $\varepsilon_{j}=1$ if $j$ is even and $0$ otherwise.

The series for $f$ in polar coordinates is $f(\rho,\Omega)=\rho^{\nu}\sum_{j=0}^{N}a_{j}\Omega^{j}\rho^{j},$ so that $f_{j}(\Omega)=a_{j}\Omega^{j},$ which is equal to $a_{j}$ if $j$ is even and $a_{j}\Omega$ if $j$ is odd. Similarly, $g_{j}(\Omega)=b_{j}$ if $j$ is even and $b_{j}\Omega$ if $j$ is odd. Using these facts, and also the observation that the power of $\Omega$ in $f_{m}^{(r)}$ is $m$, we obtain from Theorem \ref{thm:main} that
\[
\zeta_{j}=\tfrac{1}{\nu}\,\Gamma\left(  \tfrac{j}{\nu}\right)  a_{0}^{-(j+1)/\nu}2\varepsilon_{j}\left(  \sum_{m=0}^{j}b_{j-m}\sum_{r=1}^{m}\binom{-\frac{j+1}{\nu}}{r}a_{0}^{-r}a_{m}^{(r)}\right)
\]
from which it is clear that $\zeta_{j}=0$ for $j$ odd, whence we obtain the desired result.
\end{proof}

\bigskip

It is worth mentioning that in \cite[Sec. 6.1]{WojdyloSIAM}, to check the correctness of his formulas (essentially the same as Corollary \ref{cor:1D}), Wojdylo applies his expansion to the $\Gamma$-function by writing
\begin{equation}
k!=\Gamma(k+1)=\int_{0}^{\infty}e^{-t}t^{k}dt=k^{k+1}e^{-k}\int_{-1}^{\infty}e^{-k(x-\ln(x+1))}dx.\label{eqn:kfact}
\end{equation}
The resulting expansion (i.e., applying Corollary \ref{cor:1D} to (\ref{eqn:kfact})) can be used to obtain the following beautiful expression for the Stirling series (\textit{cf.} \cite[6.1.37]{Abramowitz-Stegun},
\cite{Stirling})

\[
k!\sim k^{k}e^{-k}\sqrt{2\pi k}\left(  1+\frac{1}{12k}+\frac{1}{288k^{2}}-\frac{139}{51840k^{3}}-\frac{571}{2488320k^{4}}+\frac{163879}{209018880k^{5}}+\cdots\right),~k\rightarrow\infty.
\]

\begin{corollary}
The Stirling series is given by
\[
k!\sim k^{k}e^{-k}\sqrt{2\pi k}\sum_{j=0}^{\infty}k^{-j}\sum_{r=0}^{2j}\frac{(-1)^{r}}{r!}(2j+2r-1)!!\,a_{2j}^{(r)}%
\]
where $a_{m}^{(r)}$ is the sum of all ordered products of $r$ terms of the set $\{a_{1},a_{2},\dots\}$ in which the subscripts add to $m$, where $a_{j}=\frac{1}{j+2},$ and where $(-1)!!$ is understood to be $0$.
\end{corollary}

%%%%%%%%%%%%%%%%%%%%%%%%%%%%%%%%%%%%%%%%%%%%%%%%%%%%%%%%%%%%%%%%%%
\section{Appendix}
%%%%%%%%%%%%%%%%%%%%%%%%%%%%%%%%%%%%%%%%%%%%%%%%%%%%%%%%%%%%%%%%%%

Listed below are the first few coefficients appearing in the asymptotic expansion, Theorem \ref{thm:main}:
\begin{align*}
\zeta_{0}  &  =\tfrac{1}{\nu}\Gamma\left(  \tfrac{\lambda}{\nu}\right) \int_{S^{d-1}}\frac{g_{0}(\Omega)}{f_{0}(\Omega)^{\lambda/\nu}}\,d\Omega,\\
\zeta_{1}  &  =\tfrac{1}{\nu}\Gamma\left(  \tfrac{\lambda+1}{\nu}\right) \int_{S^{d-1}}\frac{g_{1}(\Omega)}{f_{0}(\Omega)^{\frac{\lambda+1}{\nu}}} -\tfrac{(\lambda+1)}{\nu}\frac{f_{1}(\Omega)g_{0}(\Omega)}{f_{0}\Omega)^{\frac{\lambda+1}{\nu}+1}}\,d\Omega,\\
\zeta_{2}  &  =\tfrac{1}{\nu}\Gamma\left(  \tfrac{\lambda+2}{\nu}\right)\int_{S^{d-1}}\frac{g_{2}(\Omega)}{f_{0}(\Omega)^{\frac{\lambda+2}{\nu}}} -\tfrac{\lambda+2}{\nu}\frac{g_{0}(\Omega)f_{2}(\Omega)+g_{1}(\Omega)f_{1}(\Omega)}{f_{0}(\Omega)^{1+\frac{\lambda+2}{\nu}}} +\binom{-\frac{\lambda+2}{\nu}}{2}\frac{g_{0}(\Omega)f_{1}(\Omega)^{2}}{f_{0}(\Omega)^{2+\frac{\lambda+2}{\nu}}}\,d\Omega,\\
\zeta_{3}  &  =\tfrac{1}{\nu}\Gamma\left(  \tfrac{\lambda+3}{\nu}\right)\int_{S^{d-1}}\frac{g_{3}(\Omega)}{f_{0}(\Omega)^{\frac{\lambda+3}{\nu}}} -\tfrac{\lambda+3}{\nu}\frac{g_{2}(\Omega)f_{1}(\Omega)+g_{1}(\Omega)f_{2}(\Omega)+g_{0}(\Omega)f_{3}(\Omega)}{f_{0}(\Omega)^{\frac{\lambda+3}{\nu}+1}}\\
&  \qquad\qquad+\binom{-\frac{\lambda+3}{\nu}}{2}\frac{g_{1}(\Omega)f_{1}(\Omega)^{2} +2g_{0}(\Omega)f_{1}(\Omega)f_{2}(\Omega)}{f_{0}(\Omega)^{\frac{\lambda+3}{\nu}+2}}+\binom{-\frac{\lambda+3}{\nu}}{3} \frac{g_{0}(\Omega)f_{1}(\Omega)^{3}}{f_{0}(\Omega)^{\frac{\lambda+3}{\nu}+3}}\,d\Omega,\\
\zeta_{4}  &  =\tfrac{1}{\nu}\Gamma\left(  \tfrac{\lambda+4}{\nu}\right) \int_{S^{d-1}}\frac{g_{4}(\Omega)}{f_{0}(\Omega)^{\frac{\lambda+4}{\nu}}} -\tfrac{\lambda+4}{\nu}\frac{g_{0}(\Omega)f_{4}(\Omega)+g_{1}(\Omega)f_{3}(\Omega)+g_{2}(\Omega)f_{2}(\Omega) +g_{3}(\Omega)f_{1}(\Omega)}{f_{0}(\Omega)^{\frac{\lambda+4}{\nu}+1}}\\
&  \qquad\qquad+\binom{-\frac{\lambda+4}{\nu}}{2}\frac{f_{2}(\Omega)^{2}+2g_{0}(\Omega)f_{1}(\Omega)f_{3}(\Omega)+2g_{1}(\Omega)f_{1}(\Omega) f_{2}(\Omega)+g_{2}(\Omega)f_{1}(\Omega)^{2}}{f_{0}(\Omega)^{\frac{\lambda+4}{\nu}+2}}\\
&  \qquad\qquad+\binom{-\frac{\lambda+4}{\nu}}{3}\frac{3g_{0}(\Omega)f_{1}(\Omega)^{2}f_{2}(\Omega)+g_{1}(\Omega) f_{1}(\Omega)^{3}}{f_{0}(\Omega)^{\frac{\lambda+4}{\nu}+3}}+\binom{-\frac{\lambda+4}{\nu}}{4} \frac{g_{0}(\Omega)f_{1}(\Omega)^{4}}{f_{0}(\Omega)^{\frac{\lambda+4}{\nu}+4}}\,\,d\Omega.
\end{align*}

Next, we give the first two nonzero coefficients appearing in Theorem \ref{thm:best}. Here, the coefficients are expressed in terms of the derivatives of $f$ and $g$ at $\mathbf{0}$ with respect to $\mathbf{y}=P\mathbf{x}$, where $P$ is the matrix such that $H_{y}f(\mathbf{0})=Id.$ In particular, if $H_xf(\mathbf{0})=Q^{-1}DQ$ for a diagonal matrix $D$ and orthogonal matrix $Q$, then $P=\sqrt{D}Q.$ Here, $\zeta_{j}=0$ for odd values of $j$ (see Remark 2 following Theorem \ref{thm:main} or the proof of Theorem \ref{thm:best}), and we give only the nonzero values.
\begin{align*}
\zeta_{0}  &  =\frac{(2\pi)^{d/2}g(\mathbf{0})}{\sqrt{\det H_{x}f(\mathbf{0})}} \qquad \qquad \text{(Laplace's Approximation)}\\
\zeta_{2}  &  =\frac{(2\pi)^{d/2}}{\sqrt{\det}H_{x}f(\mathbf{0})}\Bigg(\sum_{\left\vert \beta\right\vert =2}\frac{D_{y}^{\beta}g}{\beta!} -\sum_{\left\vert \beta\right\vert =1}D_{y}^{\beta}g\sum_{\left\vert\alpha\right\vert =3}\frac{D_{y}^{\alpha}f}{\alpha!}(\beta+\alpha -\mathbf{1})!!\operatorname*{even}(\beta+\alpha)\\
&  \qquad+g\Big(  -\sum_{\left\vert \alpha\right\vert =4}\frac{D_{y}^{\alpha}f}{\alpha!}(\alpha-\mathbf{1})!! \operatorname*{even}(\alpha)+\tfrac{1}{2}\sum_{\left\vert \alpha_{1}\right\vert =\left\vert \alpha_{2}\right\vert=3}\frac{D_{y}^{\alpha_{1}}f~D_{y}^{\alpha_{2}}f}{a_{1}!\alpha_{2}!} (\alpha_{1}+\alpha_{2}-\mathbf{1})!!\operatorname*{even}(\alpha_{1}+\alpha_{2})\Big)  \Bigg)_{\mathbf{y}=\mathbf{0}}
\end{align*}
Finally, we give the above coefficients precisely for small $d$. (Recall that for $d=1$, our results coincide with Corollary \ref{cor:1D}, originally due to Wojdylo \cite{WojdyloSIAM}, \cite{WojdyloJCAM}.)

\begin{landscape}
\medskip\textbf{Dimension }$d=1$
\begin{align*}
\zeta_{0}  &  =\frac{\sqrt{2\pi}g(0)}{\sqrt{f^{\prime\prime}(0)}}\\
\zeta_{2}  &  =\frac{\sqrt{2\pi}}{(f^{\prime\prime}(0))^{7/2}} \left(\frac{1}{2}g^{\prime\prime}(0)(f^{\prime\prime}(0))^2-\frac{1}{2}g^{\prime}(0)f^{(3)}(0)f^{\prime\prime}(0) -\frac{1}{8}g(0)f^{(4)}(0)f^{\prime\prime}(0)+\frac{5}{24}g(0)(f^{(3)}(0))^{2}\right) \\
\zeta_{4}  &  =\frac{\sqrt{2\pi}}{(f^{\prime\prime}(0))^{13/2}} \Bigg(\frac{1}{24}g^{(4)}(0)(f''(0))^4-\frac{5}{12}g^{(3)}(0)f^{(3)}(0)(f''(0))^3 \\
& \qquad +\frac{1}{2} g''(0) \left(-\frac{5}{8}f^{(4)}(0)(f''(0))^3+\frac{35}{4}(f^{(3)}(0))^{2}(f''(0))^4\right) +g'(0) \left(-\frac{1}{8}f^{(5)}(0)(f''(0))^3 +\frac{35}{48}f^{(4)}(0)f^{(3)}(0)(f''(0))^2 -\frac{35}{48}(f^{(3)}(0))^{3}f''(0)\right) \\
&  \qquad+g(0)\Big(-\frac{1}{48}f^{(6)}(0)(f''(0))^3 +\frac{35}{384}(f^{(4)}(0))^{2}(f''(0))^2 +\frac{7}{48}f^{(5)}(0)f^{(3)}(0)(f''(0))^2 -\frac{35}{64}f^{(4)}(0)(f^{(3)}(0))^{2}f''(0)+\frac{385}{1152}(f^{(3)}(0))^{4}\Big)  \Bigg)
\end{align*}

\medskip
\textbf{Dimension }$d=2$\textbf{:}
\begin{align*}
\zeta_{0}  &  =\frac{2\pi g(\mathbf{0})}{\sqrt{\text{$\det H_{x}f(\mathbf{0})$}}}\\
\zeta_{2}  &  =\frac{2\pi g}{\sqrt{\text{$\det H_{x}f(\mathbf{0})$}}} \Bigg[\left(  \frac{1}{2}\left(  \partial_{y^{1}}^{2}g+\partial_{y^{2}}^{2}g\right)  \right)  -\frac{1}{2}\partial_{y^{1}}g\left(  \partial_{y^{1}}\partial_{y^{2}}^{2}f+\partial_{y^{1}}^{3}f\right)  -\frac{1}{2}\partial_{y^{2}}g\left(  \partial_{y^{1}}^{2}\partial_{y^{2}}f+\partial_{y^{2}}^{3}f\right) \\
&  +g\left(  -\frac{1}{8}\left(  \partial_{y^{1}}^{4}f+2\partial_{y^{1}}^{2}\partial_{y^{2}}^{2}f+\partial_{y^{2}}^{4}f\right)  +\frac{1}{4}\left(\partial_{y^{2}}^{3}f~\partial_{y^{1}}^{2}\partial_{y^{2}}f+\partial_{y^{1}}^{3}f~\partial_{y^{1}}\partial_{y^{2}}^{2}f\right)  +\frac{3}{8}\left( \left(  \partial_{y^{1}}\partial_{y^{2}}^{2}f\right)  ^{2}+\left( \partial_{y^{1}}^{2}\partial_{y^{2}}f\right)  ^{2}\right)  +\frac{5}{24}\left(  \left(  \partial_{y^{1}}^{3}f\right)  ^{2}+\left(  \partial_{y^{2}}^{3}f\right)  ^{2}\right)  \right)  \Bigg]_{\mathbf{y}=\mathbf{0}}
\end{align*}

\bigskip
\bigskip
\textbf{Dimension }$d=3$\textbf{:}
\begin{align*}
\zeta_{0} &  =\frac{(2\pi)^{3/2}g(\mathbf{0})}{\sqrt{\det H_{x}f(\mathbf{0)}}}\\
\zeta_{2} &  =\frac{(2\pi)^{3/2}}{\sqrt{\det H_{x}f(\mathbf{0)}}} \Bigg[\frac{1}{2}\left(  \partial_{y^{1}}^{2}g+\partial_{y^{2}}^{2} g+\partial_{y^{3}}^{2}g\right)  -\frac{1}{2}\partial_{y^{1}}g\left(\partial_{y^{1}}^{3}f+\partial_{y^{1}}\partial_{y^{2}}^{2}f+\partial_{y^{1} }\partial_{y^{3}}^{2}f\right)  -\frac{1}{2}\partial_{y^{2}}g\left(\partial_{y^{2}}^{3}f+\partial_{y^{2}}\partial_{y^{3}}^{2}f+\partial_{y^{1}}^{2}\partial_{y^{2}}f\right)  \\
&  -\frac{1}{2}\partial_{y^{3}}g\left(  \partial_{y^{3}}^{3}f+\partial_{y^{2}}^{2}\partial_{y^{3}}f+\partial_{y^{1}}^{2}\partial_{y^{3}}f\right)  \\
&  + g\Big(-\frac{1}{8}\left(  \partial_{y^{1}}^{4}f+\partial_{y^{2}}^{4}f+\partial_{y^{3}}^{4}f+2\partial_{y^{1}}^{2}\partial_{y^{2}}^{2}f +2\partial_{y^{1}}^{2}\partial_{y^{3}}^{2}f+2\partial_{y^{2}}^{2} \partial_{y^{3}}^{2}f\right)  +\frac{5}{24}\left(  \left(  \partial_{y^{1}}^{3}f\right)  ^{2}+\left(  \partial_{y^{2}}^{3}f\right)  ^{2}+\left(\partial_{y^{3}}^{3}f\right)  ^{2}\right)  \\
&  +\frac{3}{8}\left(  \left(  \partial_{y^{1}}^{2}\partial_{y^{2}}f\right)^{2}+\left(  \partial_{y^{1}}^{2}\partial_{y^{3}}f\right)  ^{2}+\left( \partial_{y^{2}}^{2}\partial_{y^{3}}f\right)  ^{2}+\left(  \partial_{y^{1}}\partial_{y^{2}}^{2}f\right)  ^{2}+\left(  \partial_{y^{1}}\partial_{y^{3}}^{2}f\right)  ^{2}+\left(  \partial_{y^{2}}\partial_{y^{3}}^{2}f\right)^{2}\right)  \\
&  +\frac{1}{2}\left(  \partial_{y^{1}}^{3}f\left(  \partial_{y^{1}} \partial_{y^{2}}^{2}f+\partial_{y^{1}}\partial_{y^{3}}^{2}f\right) +\partial_{y^{2}}^{2}f\left(  \partial_{y^{1}}^{2}\partial_{y^{2}} f+\partial_{y^{2}}\partial_{y^{3}}^{2}f\right)  +\partial_{y^{3}}^{2}f\left( \partial_{y^{1}}^{2}\partial_{y^{3}}f+\partial_{y^{2}}^{2}\partial_{y^{3}}f\right)  \right)  \\
&  +\frac{1}{2}\left(  \partial_{y^{1}}\partial_{y^{3}}^{2}f~\partial_{y^{1}}\partial_{y^{2}}^{2}+\partial_{y^{2}}^{2}\partial_{y^{3}}f ~\partial_{y^{1}}^{2}\partial_{y^{3}}f+\partial_{y^{2}}\partial_{y^{3}}^{2}f~\partial_{y^{1}}^{2}\partial_{y^{2}}f\right)  +\partial_{y^{1}}\partial_{y^{2}}\partial_{y^{3}}f\Big)\Bigg]_{\mathbf{y}=\mathbf{0}}
\end{align*}
\end{landscape}

\subsection*{Acknowledgements}

The author would like to thank the referees for calling to his attention the book \cite{Wong} and the papers \cite{Denef-Sargos}, \cite{Dostal-Gaveau}, \cite{Liakhovetski-Paris}, and \cite{Skinner}, and for several helpful comments and suggestions which greatly improved the exposition and motivated the extensions of Theorem \ref{thm:main} to Corollary \ref{cor:f0const} and Theorem \ref{thm:best}. The author would also like to thank R\'{e}mi Leclercq for several useful discussions.

\bigskip

{\small
\providecommand{\bysame}{\leavevmode\hbox to3em{\hrulefill}\thinspace}
\providecommand{\MR}{\relax\ifhmode\unskip\space\fi MR }
\providecommand{\MRhref}[2]{%
  \href{http://www.ams.org/mathscinet-getitem?mr=#1}{#2}
}
\providecommand{\href}[2]{#2}

}


\begin{thebibliography}{Woj06b}

\bibitem[AS64]{Abramowitz-Stegun}
M. Abramowitz and I.~A. Stegun, \emph{Handbook of mathematical functions
  with formulas, graphs, and mathematical tables}, National Bureau of Standards
  Applied Mathematics Series, vol.~55, For sale by the Superintendent of
  Documents, U.S. Government Printing Office, Washington, D.C., 1964.

\bibitem[BH75]{Bleistein-Handelsman}
N. Bleistein and R.~A. Handelsman, \emph{Asymptotic expansions of
  integrals}, Dover Publications, Inc., New York, 1975.

\bibitem[dB81]{deBruijn}
N.~G. de~Bruijn, \emph{Asymptotic methods in analysis}, third ed., Dover
  Publications Inc., New York, 1981.

\bibitem[DG89]{Dostal-Gaveau}
M. Dost{\'a}l and Bernard Gaveau, \emph{The stationary phase method
  for certain degenerate critical points. {I}}, Canad. J. Math. \textbf{41}
  (1989), no.~5, 907--931.

\bibitem[DS92]{Denef-Sargos}
J. Denef and P. Sargos, \emph{Poly\`edre de {N}ewton et distribution
  {$f\sp s\sb +$}. {II}}, Math. Ann. \textbf{293} (1992), no.~2, 193--211.

\bibitem[Fra57]{Frame}
J.~S. Frame, \emph{Power series expansions for inverse functions}, Amer. Math.
  Monthly \textbf{64} (1957), 236--240.

\bibitem[FS61]{Fulks-Sather}
W.~Fulks and J.~O. Sather, \emph{Asymptotics. {II}. {L}aplace's method for
  multiple integrals}, Pacific J. Math. \textbf{11} (1961), 185--192.

\bibitem[GR96]{Gradshteyn-Ryzhik-5th}
I.~S. Gradshteyn and I.~M. Ryzhik, \emph{Table of integrals, series, and
  products}, fifth ed., Academic Press Inc., San Diego, CA, 1996, CD-ROM
  version 1.0 for PC, MAC, and UNIX computers.

\bibitem[GS82]{Guillemin-Sternberg}
V. Guillemin and S. Sternberg, \emph{{Geometric Quantization and
  Multiplicities of Group Representations}}, Inv. Math. \textbf{67} (1982),
  515--538.

\bibitem[HK06]{Hall-K}
B.~C. Hall and W.~D. Kirwin, \emph{Unitarity in ``quantization commutes
  with reduction''}, Comm. Math. Phys. \textbf{275} (2006), no.~3, 410--442.

\bibitem[Kam46]{Kamber}
F. Kamber, \emph{Formules exprimant les valeurs des coefficients des
  s\'eries de puissances inverses}, Acta Math. \textbf{78} (1946), 193--204.

\bibitem[KP98a]{Kaminski-ParisI}
D.~Kaminski and R.~B. Paris, \emph{Asymptotics of a class of multidimensional
  {L}aplace-type integrals. {I}. {D}ouble integrals}, R. Soc. Lond. Philos.
  Trans. Ser. A Math. Phys. Eng. Sci. \textbf{356} (1998), no.~1737, 583--623.

\bibitem[KP98b]{Kaminski-ParisII}
\bysame, \emph{Asymptotics of a class of multidimensional {L}aplace-type
  integrals. {II}. {T}reble integrals}, R. Soc. Lond. Philos. Trans. Ser. A
  Math. Phys. Eng. Sci. \textbf{356} (1998), no.~1737, 625--667.

\bibitem[Lap95]{Laplace-Prob}
P.-S. Laplace, \emph{Th\'eorie analytique des probabilit\'es. {V}ol.
  {I}}, \'Editions Jacques Gabay, Paris, 1995, Introduction: Essai
  philosophique sur les probabilit{\'e}s. [Introduction: Philosophical essay on
  probabilities], Livre I: Du calcul des fonctions g{\'e}n{\'e}ratrices. [Book
  I: On the calculus of generating functions], Reprint of the 1819 fourth
  edition (Introduction) and the 1820 third edition (Book I).

\bibitem[LP01]{Liakhovetski-Paris}
G.~V. Liakhovetski and R.~B. Paris, \emph{Asymptotic expansions of
  {L}aplace-type integrals. {III}}, J. Comput. Appl. Math. \textbf{132} (2001),
  no.~2, 409--429.

\bibitem[Ski80]{Skinner}
L.~A. Skinner, \emph{Note on the asymptotic behavior of multidimensional
  {L}aplace integrals}, SIAM J. Math. Anal. \textbf{11} (1980), no.~5,
  911--917.

\bibitem[Sti30]{Stirling}
J.~Stirling, \emph{Methodus differentialis, sive tractatus de summation et
  interpolation serierum infinitarium}, London, 1730, English translation by
  Holliday, J. The Differential Method: A Treatise of the Summation and
  Interpolation of Infinite Series. 1749.

\bibitem[Woj06a]{WojdyloSIAM}
J. Wojdylo, \emph{Computing the coefficients in {L}aplace's method}, SIAM
  Rev. \textbf{48} (2006), no.~1, 76--96 (electronic).

\bibitem[Woj06b]{WojdyloJCAM}
\bysame, \emph{On the coefficients that arise from {L}aplace's method}, J.
  Comput. Appl. Math. \textbf{196} (2006), no.~1, 241--266.

\bibitem[Won01]{Wong}
R.~Wong, \emph{Asymptotic approximations of integrals}, Classics in Applied
  Mathematics, vol.~34, Society for Industrial and Applied Mathematics (SIAM),
  Philadelphia, PA, 2001, Corrected reprint of the 1989 original.

\end{thebibliography}
\end{document}